\documentclass[12pt]{amsart}
\usepackage{amsfonts}
\usepackage{fullpage, amsmath, amsthm,amsfonts,amssymb,stmaryrd, mathrsfs}
\usepackage{hyperref}

\usepackage[pdftex]{color}
\usepackage[all]{xy}

\newtheorem{thm}{Theorem}

\numberwithin{equation}{subsection}
\newtheorem{theorem}[subsection]{Theorem}
\newtheorem{lemma}[subsection]{Lemma}
\newtheorem{corollary}[subsection]{Corollary}
\newtheorem{conjecture}[subsection]{Conjecture}

\newtheorem{proposition}[subsection]{Proposition}
\theoremstyle{definition}
\newtheorem{caution}[subsection]{Caution}

\newtheorem{hypothesis}[subsection]{Hypothesis}
\newtheorem{example}[subsection]{Example}

\newtheorem{remark}[subsection]{Remark}

\newtheorem{notation}[subsection]{Notation}

\def\calA{\mathcal{A}}
\def\calB{\mathcal{B}}

\def\calO{\mathcal{O}}

\def\calS{\mathcal{S}}

\def\calW{\mathcal{W}}

\def\gothC{\mathfrak{C}}
\def\gothD{\mathfrak{D}}

\def\gothP{\mathfrak{P}}

\def\gothT{\mathfrak{T}}
\def\gothU{\mathfrak{U}}

\def\AAA{\mathbb{A}}

\def\CC{\mathbb{C}}
\def\FF{\mathbb{F}}
\def\GG{\mathbb{G}}

\def\NN{\mathbb{N}}

\def\QQ{\mathbb{Q}}
\def\RR{\mathbb{R}}

\def\ZZ{\mathbb{Z}}

\def\bfe{\mathbf{e}}

\def\rmM{\mathrm{M}}

\def\scrB{\mathscr{B}}

\def\i{\mathbf{i}}
\def\j{\mathbf{j}}
\def\k{\mathbf{k}}

\DeclareMathOperator{\HP}{HP}

\DeclareMathOperator{\Max}{Max}
\DeclareMathOperator{\NP}{NP}

\DeclareMathOperator{\Spf}{Spf}

\newcommand{\new}{\textrm{-}\mathrm{new}}

\newcommand{\AL}{\mathrm{AL}}
\newcommand{\ord}{\mathrm{ord}}
\newcommand{\Nm}{\mathrm{Nm}}

\newcommand{\Qp}{\QQ_p}
\newcommand{\rig}{\mathrm{rig}}

\newcommand{\wt}{\mathrm{wt}}

\newcommand{\Zp}{\ZZ_p}

\newcommand{\cl}{\mathrm{cl}}
\newcommand{\Err}{\mathbf{Err}}
\newcommand{\Matrix}[4]{\big( \begin{smallmatrix}
#1&#2\\#3&#4
\end{smallmatrix}\big)}
\newcommand{\MATRIX}[4]{ \begin{pmatrix}
#1&#2\\#3&#4
\end{pmatrix}}

\newcommand{\Frob}{\mathrm{Frob}}
\DeclareMathOperator{\GL}{GL}

\DeclareMathOperator{\Char}{Char}
\DeclareMathOperator{\Spc}{Spc}
\DeclareMathOperator{\Diag}{Diag}

\begin{document}

\title{Slopes of eigencurves over boundary disks}
\author{Daqing Wan}
\address{Daqing Wan,
University of California at Irvine,
Department of Mathematics, 340 Rowland Hall, Irvine, CA 92697, U.S.A.}
\email{dwan@math.uci.edu}
\author{Liang Xiao}
\address{Liang Xiao,
University of Connecticut, Department of Mathematics, 196 Auditorium Road, Unit 3009, Storrs, CT 06269--3009, U.S.A.}
\email{liang.xiao@uconn.edu}
\author{Jun Zhang}
\address{Jun Zhang, School of Mathematical Sciences, Capital Normal University, Beijing 100048, P.R. China.}
\email{junz@cnu.edu.cn}
\date{\today}
\begin{abstract}
Let $p$ be a prime number.
We study the slopes of $U_p$-eigenvalues on the subspace of modular forms that can be transferred to a definite quaternion algebra.
We give a sharp lower bound of the corresponding Newton polygon.
The computation happens over a definite quaternion algebra by Jacquet--Langlands correspondence; it generalizes a prior work of Daniel Jacobs \cite{jacobs} who treated the case of $p=3$ with a particular level.

In case when the modular forms have a finite character of conductor highly divisible by $p$, we improve the lower bound to show that the slopes of $U_p$-eigenvalues grow roughly like arithmetic progressions as the weight $k$ increases.  This is the first very positive evidence for Buzzard--Kilford's conjecture on the behavior of the eigencurve near the boundary of the weight space, that is proved for arbitrary $p$ and general level.
We give the exact formula of a fraction of the slope sequence.

\end{abstract}
\maketitle

\setcounter{tocdepth}{1}
\tableofcontents

\section{Introduction}
Let $p$ be a fixed prime number which we assume to be odd for simplicity in this introduction.
For $N$ a positive integer (the ``tame level") coprime to $p$, $k+1 \geq 2$ an integer (the ``weight")\footnote{For the subject we study, writing $k+1$ for the weight will simplify the presentation.}, $m$ a positive integer, and $\psi$ a character of $(\ZZ/p^m\ZZ)^\times$, we use $S_{k+1}(\Gamma_0(p^mN); \psi)$ to denote the space of modular cuspforms of weight $k+1$,  level $p^mN$, and nebentypus character $\psi$ over some finite extension $E$ of $\QQ_p$.
This space comes equipped with the action of Hecke operators, most importantly the action of the $U_p$-operator.
It is a central question in the theory of $p$-adic modular forms to understand the distributions of the ``slopes", namely, the $p$-adic valuations of the eigenvalues of $U_p$ acting on $S_{k+1}(\Gamma_0(p^mN); \psi)$, as the weight $k+1$ varies.
\emph{All $p$-adic valuations or norms in this paper are normalized so that $p$ has valuation $1$ and norm $p^{-1}$.}


One of the most interesting expectations concerns the case when the nebentypus character  $\psi$ has exact conductor $p^m$ for $m \geq 2$, i.e. $\psi$ does not factor through a character on $(\ZZ/p^{m-1}\ZZ)^\times$.
Let $\omega: (\ZZ/p\ZZ)^\times \to \ZZ_p^\times$ denote the Teichmu\"uller character.

The following question was asked by Coleman and Mazur \cite{coleman-mazur} and later elaborated by Buzzard and Kilford \cite{buzzard-kilford}.

\begin{conjecture}
\label{Conj:weak-buzzard-kilford}
Fix an integer $N$ coprime to $p$ and a character $\psi_0$ of $(\ZZ/p\ZZ)^\times$ such that $\psi_0(-1) =-1$.
Then
there exists a non-decreasing sequence of rational numbers $a_1, a_2, \dots$ approaching to infinity such that
\begin{itemize}
\item
for any integers $m\geq 2, k+1 \geq 2$, and any character $\psi$ of $(\ZZ/p^m\ZZ)^\times$ of exact conductor $p^m$ such that $\psi|_{(\ZZ/p\ZZ)^\times}\cdot \omega^k = \psi_0$, the slopes of $U_p$ acting on $S_{k+1}(\Gamma_0(p^mN); \psi)$
 is given by the first few terms of the sequence
\[
a_1/p^m,\ a_2 / p^m,\ \dots
\]
consisting of all numbers strictly less than $k$ and some equal to $k$'s.
\end{itemize}
Moreover, the sequence $a_1, a_2, \dots$ is a union of finitely many arithmetic progressions.
\end{conjecture}


There has been many direct computations supporting this Conjecture in special cases, first by Buzzard and Kilford \cite{buzzard-kilford} (extending the work of Emerton \cite{emerton}) in the case when $p=2$ and $N=1$\footnote{We earlier excluded the case of $p=2$ for simple presentation; but slight modification allows us to include this case, as we will do for the rest of the paper.}, then in  many similar particular cases with small primes $p$ and small levels; see
\cite{roe, kilford, kilford-mcmurdy, jacobs}.
Nonetheless, this Conjecture was never recorded in the literature for lack of theoretic or heuristic evidences.
The goal of this paper is to provide some positive indications in the general case.

\subsection{The geometry of the eigencurve}

Before proceeding, we explain the meaning of Conjecture~\ref{Conj:weak-buzzard-kilford} in terms of the geometry of the eigencurve.

Eigencurves were introduced by Coleman and Mazur \cite{coleman-mazur} to $p$-adically interpolate modular eigenforms of different weights.
Here the notion of weights is generalized to mean a continuous character of $\ZZ_p^\times$; for examples, $x \mapsto x^k \psi(x)$ corresponds the case of classical weight ${k+1}$ with nebentypus character $\psi$.
In the loosest terms, the eigencurve
is a rigid analytic closed subscheme of  the product of the weight space and $\GG_m$, defined as the Zariski closure  of the set of pairs $(x^k\psi(x), a_p(f))$ for each eigenform $f$ of weight ${k+1}$ and nebentypus character $\psi$ with $U_p$-eigenvalue $a_p(f)$.
In particular, its fiber over the point $x^k\psi(x)$ of the weight space parametrizes the $U_p$-eigenvalues on the space of modular forms $S_{k+1}(\Gamma_0(p^mN); \psi)$ and the overconvergent ones.

The eigencurve plays a crucial role and has many applications in the modern $p$-adic number theory; to name one: Kisin's proof of Fontaine-Mazur conjecture \cite{kisin}.
Despite the many arithmetic applications, the geometry of the eigencurve was however poorly understood for a long time.
For example, the properness of the eigencurve was not known until the very recent work of  Diao and  Liu \cite{diao-liu}.

Conjecture~\ref{Conj:weak-buzzard-kilford} and this paper focus on another intriguing property: the behavior of the eigencurve near the boundary of the weight space.
The striking computation of Buzzard and Kilford \cite{buzzard-kilford} mentioned above shows that, when $p=2$ and $N=1$, the Coleman-Mazur eigencurve, when restricted over the boundary annulus of the weight space, is an infinite disjoint union of copies of this annulus.
This is  a family and a much stronger version of Conjecture~\ref{Conj:weak-buzzard-kilford}; see Conjecture~\ref{Conj:CM-conj} for the precise expectation.
Generalizing this result would have many number theoretical applications. For example, in \cite{pottharst-xiao}, the second author and Pottharst reduced the parity conjecture of Selmer rank for modular forms to this precise statement.


\subsection{Main result of this paper}
\label{SS:main result}
For the sake of presentation, we assume that there exists a prime number $\ell$ such that $\ell|| N$.
We only consider the subspace of modular forms which are $\ell$-new, denote by a superscript $\ell\new$, e.g. $S_{k+1}(\Gamma_0(p^mN); \psi)^{\ell \new}$.
This is the subspace of modular forms that can be identified by Jacquet--Langlands correspondence with the automorphic forms on a definite quaternion algebra $D$ which ramifies at $\ell$ and $\infty$.

The following lower bound of the Newton polygon of the $U_p$-action on $S_{k+1}(\Gamma_0(p^mN); \psi)^{\ell \new}$ might be known among some experts.\footnote{We think that Buzzard probably has an unpublished note on certain version of this theorem; see \cite{buzzard-slope}.}

\begin{thm}
\label{T:theorem A}
Assume that the conductor of $\psi$ is exactly $p^m$. (By our later convention, this will include the case when $\psi$ is trivial and $m=1$.)
Let $t$ denote $\dim S_2(\Gamma_0(p^mN); \psi)^{\ell \new}$ so that $\dim S_{k+1}(\Gamma_0(p^mN); \psi)^{\ell \new} = kt$.
Then the Newton polygon of the $U_p$-action on $S_{k+1}(\Gamma_0(p^mN); \psi)^{\ell \new}$ lies above the polygon with vertices
\[
(0,0), (t,0), (2t, t), \dots, (nt, \tfrac{n(n-1)}{2}t), \dots
\]
\end{thm}
The complete proof is given in Theorem~\ref{T:weak Hodge polygon}. Note that the lower bound is independent of $k$, and thus uniform in $k$.
A similar uniform quadratic lower bound of Newton polygon was obtained by the first named author in \cite{wan} using a variant of Dwork's trace formula.
Our lower bound here is very sharp: the distance between the end point of the Newton polygon of $U_p$ acting on $S_{k+1}(\Gamma_0(p^mN); \psi)^{\ell \new}$ and our lower bound is \emph{linear} in $k$.

When the character $\psi$ is trivial, Theorem~\ref{T:theorem A} gives some theoretic evidence of a conjecture of Gouv\^ea on the distributions of slopes. But the method presented here is not enough to prove this conjecture of Gouv\^ea.  We refer to Remarks~\ref{R:improve Hodge bound classical forms} and \ref{R:heuristic gouvea} for related discussions.

We also point out that Theorem~\ref{T:theorem A} may suggest a very effective way to compute the eigencurve using definite quaternion algebras; the statement implies that the computation converges very well, comparable to the prevailing method of modular symbols.

The proof of Theorem~\ref{T:theorem A} (and the proof of the subsequent theorems in this paper) uses Jacquet--Langlands correspondence to transfer all information into automorphic forms for a definite quaternion algebra.
The advantage of working with definite quaternion algebra is its simpler geometry compared to the modular curves.
The theory of overconvergent automorphic forms on a definite quaternion algebra d'apr\`es Buzzard \cite{buzzard} come equipped with a nice integral basis.
Our computation essentially reproduces  Jacobs' thesis \cite{jacobs}, except taking a more theoretical as opposed to computational approach.

The real improvement over Jacobs' work is that, when the conductor $p^m$ of $\psi$ is large (e.g. $m \geq 4$), we can improve the lower bound above so that it \emph{agrees} with the Newton polygon (in the overconvergent setting) at infinitely many points which form an arithmetic progression.  This gives the following
\begin{thm}
\label{T:theorem B}
Keep the notation as in Theorem~\ref{T:theorem A} and assume that $m \geq 4$.
Let $a_0(k) \leq a_1(k) \leq \dots \leq a_{kt-1}(k)$ denote the slopes of the $U_p$-action on $S_{k+1}(\Gamma_0(p^mN); \psi)^{\ell \new}$, in non-decreasing order (with multiplicity).
Then we have
\[
 \lfloor \tfrac nt\rfloor \leq a_n(k)\leq \lfloor \tfrac nt\rfloor +1.
\]
\end{thm}

This is proved in Theorem~\ref{T:sharp Hodge bound}. Note that the inequality of the slopes does not depend on the weight $k+1$.
In fact, we prove a family version of such inequality which gives rise to a  decomposition (Theorem~\ref{T:improved main theorem}) of the eigencurve over the disks $\calW(x\psi, p^{-1})$ of radius $p^{-1}$ centered around the character $x\psi$, just as in Buzzard--Kilford \cite{buzzard-kilford}.  Unfortunately, we cannot extend this result to the entire weight annulus of radius $p^{-1/p^{m-2}(p-1)}$ which contains $x\psi$.\footnote{In recent joint work of the first two authors and Ruochuan Liu \cite{liu-wan-xiao}, we extend this result to the entire boundary of the weight space, through using a different basis for the overconvergent automorphic forms.  Many ideas of \cite{liu-wan-xiao} are taken from this paper.}

The main idea of the proof consists of two major inputs: (1) We show that there is a natural isomorphism
\begin{equation}
\label{E:overconvergent = classical}
S^{D, \dagger}(U; \kappa) \cong \widehat\bigoplus_{n=0}^\infty S_2^D(U; \psi \omega^{-2n}) \otimes (\omega^n \circ \det),
\end{equation}
such that the $U_p$-action on the left hand side is  ``approximately" the action of $\bigoplus_{n\geq 0} (p^n \cdot U_p)$ on the right hand side.
Here the letter $U$ is the corresponding level structure which looks like $\Gamma_0(p^m)$ at $p$;
$S^{D, \dagger}(U; \kappa)$
stands for the space of overconvergent automorphic forms over a definite quaternion algebra $D$ with weight character $\kappa$ living in $\calW(x\psi, p^{-1})$;
the right hand side is the completed direct sum of \emph{classical} automorphic forms over $D$ of \emph{weight $2$} with characters $\psi \omega^{-2n}$, \emph{twisted by the character $\omega^n \circ \det$}.
It thus follows that the $U_p$-slopes on $S^{D, \dagger}(U; \kappa)$ is approximately determined by the $U_p$-slopes on  these space of classical forms of weight $2$.

(2) To carry out the approximation in \eqref{E:overconvergent = classical}, it is important to show that the slopes of the \emph{Hodge polygon} of the $U_p$-action on each $S_2^D(U; \psi\omega^{-2n})$ are between $0$ and $1$. Here the Hodge polygon
of the matrix for  the $U_p$-action
 refers to the convex hull of points given by the minimal $p$-adic valuation of the minors of the matrix.
To prove this key result, we make use of (in the definite quaternion situation) the Atkin--Lehner map (see \ref{S:Atkin-Lehner})
\[
\AL_{\psi_m}: S_2^D(U; \psi_m) \longrightarrow S_2^D(U; \psi_m^{-1}) \otimes (\psi_m \circ \det)
\]
and the fact that $U_p\circ\AL_{\psi_m}\circ U_p(\varphi) = p \cdot \AL_{\psi_m}\circ S_p(\varphi)$, where $S_p$ is the unramified central character action at $p$.
In fact, we also need certain deformed version of this map in order to improve the result from the open disks of radius $p^{-1}$ to the closed disks of the same radius.  This small improvement is essential to Theorem~\ref{T:theorem B}.
We refer to Section~\ref{Section:improve Hodge polygon} for details.

We also point out that the condition $m\geq 4$ is currently an unfortunate technical condition.
See Remark~\ref{R:m=3} for the discussion in the case when $m=3$.

A consequence of the proof of Theorem~\ref{T:theorem B} is that we can in fact show that some of the slopes indeed form arithmetic progressions.

\begin{thm}
\label{T:theorem C}
Keep the notation as in Theorem~\ref{T:theorem B}.
Fix $r \in \{0,1, \dots, \frac{p-3}{2}\}$.
Let $\NP_r(i)$ and $\HP_r(i)$ denote the Newton polygon and Hodge polygon functions for the $U_p$-action on $S_2(\Gamma_0(p^mN); \psi \omega^{-2r})^{\ell \new}$.
Suppose that $(s_0, \NP_r(s_0))$ is a vertex of the Newton polygon $\NP_r$ and suppose that
\[
\NP_r(s) < \HP_r(s-1) + 1 \textrm{ for all }s =1, \dots, s_0.
\]
Then for any $s = 0,1, \dots, s_0$, the following subsequence
\[
a_{s+rt}(k),\ a_{s+rt+\frac{p-1}{2}t}(k),\ \dots, \  a_{s+ rt+ i \frac{p-1}2t}(k),\ \dots
\]
is independent of the positive integer $k$ whenever $k \equiv 2r+1 \bmod{p-1}$ (and whenever it makes sense) and it forms an arithmetic progression with common difference $\frac{p-1}{2}$.
\end{thm}

This is proved in Corollary~\ref{C:precise NP computation}.
\emph{Note that the common difference for the arithmetic progression is $\frac{p-1}{2}$ but not $1$}.
This is due to the periodic appearance of the powers of Teichm\"uller characters in \eqref{E:overconvergent = classical}.
In fact, this (larger) common difference agrees with the computation of Kilford \cite{kilford} and Kilford--McMurdy \cite{kilford-mcmurdy} in the case $m=2$, where the common difference is $2$ when $p=5$ and is $\frac 32$ (which can be further separated into two arithmetic progressions with common difference $3$) when $p=7$.

The power of Theorem~\ref{T:theorem C}
is limited by how close the Hodge polygon is to the Newton polygon.
In particular, as $N$ and $m$ get bigger, the gap between the Newton and Hodge polygons will be inevitably widened, and hence $s_0$ is relatively small compared to $t$.

One remedy we propose is to ``decompose" the space of (overconvergent) modular forms according to the associated residual Galois (pseudo-)representations.\footnote{Galois pseudo-representations are equivalent to semisimple Galois representations. Since we are really using the tame Hecke eigenvalues, we prefer to use the concept of pseudo-representations.}

\begin{thm}
\label{T:theorem D}
Let $\bar \rho_1, \dots, \bar \rho_d$ be the residual Galois pseudo-representations appearing as the pseudo-representations attached to the eigenforms in $S_2^D(U; \psi \omega^{-2r})$ for some $r =0,1, \dots, \frac{p-3}{2}$.
Then we have a natural decomposition of (overconvergent) automorphic forms:
\[
S^{D, \dagger}(U; \kappa) = \bigoplus_{j=1}^d S^{D, \dagger}(U; \kappa)_{\bar \rho_j} \quad \textrm{and}\quad
S_{k+1}^{D}(U; \psi) = \bigoplus_{j=1}^d S_{k+1}^{D}(U; \psi)_{\bar \rho_j}
\]
for all weights $k+1$.
Moreover, Theorem~\ref{T:theorem C} holds for each individual $S_2^D(U, \psi \omega^{-2r})_{\bar \rho_j}$.
\end{thm}

This is proved in Theorem~\ref{T:main theorem each residual}.  The idea behind this theorem is that the isomorphism \eqref{E:overconvergent = classical} is also approximately equivariant for the tame Hecke actions.  One can certainly decompose the right hand side of  \eqref{E:overconvergent = classical} according to the reductions of the associated Galois (pseudo-)representations; the isomorphism \eqref{E:overconvergent = classical} allows us, to some extend, transfer the decomposition to the space of overconvergent automorphic forms.  The error terms can be killed by taking  the limit of repeated $p$-powers of the approximate projectors on the space of overconvergent automorphic forms.

We believe that the decomposition by Galois pseudo-representations has its own interest; for example, it gives a natural decomposition of the eigencurve according to the residual Galois pseudo-representations.
Our decomposition is given in a reasonably explicit way on the Banach space of overconvergent automorphic forms and we have a good ``model" of each factor.
So the decomposition of the eigencurve over disks of radius $p^{-1}$ centered around $x \psi(x)$ applies to the piece corresponding to each Galois pseudo-representation.

\subsection{Relation with later works}

Recently, the first two authors and R. Liu \cite{liu-wan-xiao} proved many cases of Conjecture~\ref{Conj:CM-conj} of Coleman--Mazur and Buzzard--Kilford. The method is very similar to this paper, but made use of a difference basis for automorphic forms.

\subsection{Structure of the paper}
\label{S:structure of paper}

We first briefly recall the construction of eigencurves in Section~\ref{Section:CM eigencurve} as well as the conjecture of Coleman--Mazur and Buzzard--Kilford.
Section~\ref{Section:automorphic forms} sets up basic notations for classical and overconvergent automorphic forms for a definite quaternion algebra.  
Section~\ref{Section:computation of Up} gives the most fundamental computation of the infinite matrix for the $U_p$-action on the space of overconvergent automorphic forms.
In particular, Theorem~\ref{T:theorem A} is proved here.
The theoretical computation is complemented by a concrete example which we present in Section~\ref{Section:explicit example}; this was previously studied by Jacobs \cite{jacobs} who relies heavily on computer computation, but made much more accessible here as a by-hand computation.  We hope this explicit example can inspire the readers to seek for new ideas.
After this, we study the Atkin--Lehner involution in Section~\ref{Section:improve Hodge polygon} and prove Theorems~\ref{T:theorem B} and \ref{T:theorem C} at the end of the section.
Section~\ref{Section:separation residue} is devoted to separating the eigencurve according to residual Galois pseudo-representations.
Theorems~\ref{T:theorem D} is proved at the end of Section \ref{Section:separation residue}.

\subsection*{Acknowledgments}
We are grateful to Frank Calegari, Matthew Emerton, Chan-Ho Kim, and Xinyi Yuan for many useful discussions. We thank the anonymous referee for carefully reading the paper and for suggestive comments.
We thank Chris Davis and Hui June Zhu for their interests.
We thank {\tt Sage notebook} and {\tt lmfdb.org} for providing numerical input in the course of this research.
The first author is partially supported by Simons Fellowship.
The second author is  partially supported by Simons Collaboration Grant \#278433, NSF Grant DMS--1502147, and CORCL research grant from University of California, Irvine.
The third author is supported by Beijing outstanding talent training program (\#2014000020124G140). And the third author would like to thank University of California, Irvine for the hospitality during his visit.

\subsection*{Unconventional use of notations}
We list a few unconventional use of notations.
\begin{itemize}
\item
The conductor of a trivial character of $\ZZ_p^\times$ is $p$ as opposed to $1$.
\item
We use $k+1$, as opposed to $k$, for the weight of modular forms.
Related to this, the right action appearing in the definition of automorphic forms on definite quaternion algebra uses a slightly different normalization; see \eqref{E:right action}.
\item
Although the Hecke actions seem to come from certain right actions on the Tate algebras, we still view them as left actions.  Therefore, we exclusively work with column vectors.  We will try to clarify this in the context (e.g. Proposition~\ref{P:generating series}).
\item
All row and column indices of a matrix start with $0$ as opposed to $1$; this will be extremely useful when considering infinite matrices later.
\end{itemize}

\section{Coleman--Mazur eigencurves}
\label{Section:CM eigencurve}

\subsection{Weight space}
We fix a prime number $p$.
We write $\Gamma=\ZZ_p^\times$ as $\Delta \times \Gamma_0$, where $\Gamma_0 = (1+2p\ZZ_p)^\times\cong \ZZ_p$ (identified via the map $x \mapsto \frac 1{2p}\log(x) = \frac 1{2p}\big( (x-1) - \frac{(x-1)^2}2+ \cdots \big)$) and $\Delta = (\Zp/2p\Zp)^\times$ is isomorphic to $\ZZ/(p-1)\ZZ$ if $p\geq 3$,  and $\ZZ/2\ZZ$ if $p=2$.
We choose the topological generator $\gamma_0$ of $ \Gamma_0$ to be the element $\exp(2p) \in\Gamma_0 \subseteq \Zp^\times$.

We use $\Lambda = \Zp\llbracket \Gamma\rrbracket$ and $\Lambda_0
 = \Zp\llbracket\Gamma_0\rrbracket$ to denote the Iwasawa algebras.
In particular, we have $\Lambda \cong \Lambda _0 \otimes_{\Zp} \Zp[\Delta]$.  For an element $\gamma \in \Gamma$, we use $[\gamma]$ to denote its image in the Iwasawa algebra $\Lambda$.
The chosen $\gamma_0$ defines an isomorphism $\ZZ_p\llbracket T\rrbracket \simeq \Lambda _0 $ given by $T \mapsto [\gamma_0] -1$.

The weight space is defined to be $\calW: = \Max(\Lambda[\frac 1p])$, the rigid analytic space associated to the formal scheme $\Spf(\Lambda)$; it is a disjoint union of $\#\Delta$ copies of the open unit disk.  
The natural projection
\[
\calW \cong \Max(\Lambda_0 \otimes_{\Zp} \Qp[\Delta]) \to \Max(\Lambda_0[\tfrac1p]) \simeq \Max(\Zp\llbracket T\rrbracket[\tfrac 1p])
\]
gives each point on $\calW$ a \emph{$T$-coordinate}.

The weight space $\calW$ may be viewed as the universal space for continuous characters of $\Gamma$.  More precisely,  a continuous character $\kappa:\Gamma \to \calO_{\CC_p}^\times$ gives rise to a continuous homomorphism $\kappa: \Lambda = \Zp\llbracket \Gamma\rrbracket \to \calO_{\CC_p}$ and hence defines a point, still denoted by $\kappa$, on the weight space $\calW$.
The $T$-coordinate of the point $\kappa$ is
$
T_\kappa = \kappa(\gamma_0) -1.
$
We point out that, the $T$-coordinate of a point of $\calW$ depends on the choice of the topological generator $\gamma_0$, but its $p$-adic valuation does not.

\begin{example}
\label{Ex: characters on weight space}
For $k \in \ZZ$, the character $x^k: \Gamma \to \ZZ_p^\times$ sending $a$ to $a^k$ has $T$-coordinate $T_{x^k} = \exp(2kp) - 1$.  We observe that $|\exp(2kp)-1| = p^{-v_p(2kp)}$; in other words, these types of points are very closed to the centers of the weight disks.

Let $\psi_m: \Gamma \to (\Zp / p^{m} \Zp)^\times \to \calO_{\CC_p}^\times$ denote a finite continuous character which does not factor through smaller \emph{positive} $m$ ($m \geq 2$ if $p=2$). We say that $\psi_m$ has \emph{conductor $p^m$}, ignoring the prime-to-$p$ part of the conductor.
In particular, a trivial character has conductor $p$ (or $4$ if $p=2$); so $m=1$ (or $m=2$ if $p=2$).
When $m \geq 2$ and $p>2$, $\psi_m(\gamma_0)$ is a primitive $p^{m-1}$-st root of unity $\zeta_{p^{m-1}}$.  Thus, the point $x^k \psi_m$ has $T$-coordinate $\zeta_{p^{m-1}}\exp(2pk) -1$, which has norm $ p^{-1/p^{m-2}(p-1)}$ (independent of $k$).
So these points tend towards the boundary of the weight disks as $m$ increases; \emph{but stay in the same ``rim" as $k$ varies, and accumulate as $k$ becomes more congruent modulo powers of $p$.}

We call characters $x^k\psi_m$ with $k \geq 1$ \emph{classical characters}. (Our weight will always be $k+1$ from now on.)

We use $\omega: \Delta \to \ZZ_p^\times$ to denote the Teichm\"uller character.  We use $\langle\cdot \rangle: \Gamma\to \ZZ_p^\times $ to denote the character $x\omega^{-1}$.
\end{example}

\subsection{Coleman--Mazur eigencurve}
\label{S:CM eigencurve}
Instead of working with the usual eigencurves, we shall work with the so-called ``spectral curves"; the main Conjecture~\ref{Conj:CM-conj} is, for a large part, equivalent for these two curves.

We first recall the definition of spectral curves; for details, we refer to \cite[Section~2]{buzzard}.
Suppose that we are given an affinoid algebra $A$\footnote{Typically, $\Max(A)$ is an affinoid subdomain of $\calW$.} over $\QQ_p$  and a Banach $A$-module $S$ which satisfies Buzzard's property (Pr) (see \cite[after Lemma 2.10]{buzzard}), that is a Banach $A$-module isomorphic to a direct summand of a Banach $A$-module $P$ which admits a countable orthonormal basis $(e_i)_{i \in \NN}$.
Moreover, suppose that we are given a \emph{nuclear} operator $U_p$ on $S$, that is, the uniform limit of a sequence of continuous $A$-linear operators on $S$ whose images are finite $A$-modules.
Then we can extend the action of $U_p$ to the ambient space $P$ by taking the zero action on other direct summands of $P$.
Write $U_p$ as an infinite matrix $M$, respect to the basis $(e_i)$.
Then the \emph{characteristic power series} of $U_p$ acting on $S$
\[
\Char(U_p; S): = \det (I - XM) = 1+ c_1 X+ c_2 X^2 + \cdots \in A\llbracket X \rrbracket
\]
converges and is independent of the choices of the ambient space $P$ and its basis $(e_i)$.
Moreover, we have $\lim_{n\to \infty} |c_n| r^n =0$ for any $r \in \RR^+$.
Consequently, it makes sense to talk about the zero locus of the characteristic power series $\Char(U_p;S)$ in $\Max(A) \times \GG_{m, \rig}$, where $X$ is the coordinate of the second factor.
We denote this zero locus by $\Spc :=\Spc(U_p;S)$; it is called the \emph{spectral variety} associated to the Banach module $S$ and the $U_p$-operator.
The natural projection $\wt: \Spc\to \Max(A)$ is called the \emph{weight map}; the map $a_p: \Spc \to \GG_{m,\rig}\xrightarrow{x \to x^{-1}}\GG_{m, \rig}$ given by the composite of the other natural projection with an inverse map is called the  \emph{slope map}.
\[
\xymatrix{
\Spc \ar[d]^{\wt} \ar[r]^-{a_p} &\GG_{m,\rig}\\ \Max(A).
}
\]
The weight map is known to be locally finite.
For each closed point $z \in\Spc$, we use $|\wt(z)|$ to denote the absolute value of the $T$-coordinate of $z$ and $|a_p(z)|$ to denote the absolute value of the corresponding point with respect to the natural coordinate on $\GG_{m, \rig}$.

In the case of elliptic modular forms (with level $\Gamma_0(p)$), Coleman and Mazur \cite{coleman-mazur} constructed, for each affinoid subdomain $A$ of the weight space $\calW$, a Banach module $M$ consisting of overconvergent cuspidal modular forms of weight in $A$ and of a fixed convergence radius; it carries a natural action of the $U_p$-operator.
This construction was subsequently generalized by Buzzard \cite{buzzard} to allow arbitrary tame level on the modular curve.
Using the construction of the previous paragraph, one can define the spectral curve over $\Max(A)$, which patches together over $\calW$ as the subdomain $\Max(A)$ varies.
We do not recall the precise definition here, but refer to \cite{buzzard} for details.  However, we shall later encounter a slightly different situation working with definite quaternion algebras. Detailed construction of the corresponding Banach module will be given then.

\subsection{The eigencurve near the boundary of the weight space}
Recall that weight space $\calW$ has a natural coordinate $T$.
For $r <1$, we use $\calW^{\geq r}$ to denote the sub-annulus of $\calW$ where $r \leq |T| <1$, called the \emph{rim} of the weight space (after Mazur).
We are mostly interested in the situation when $r\to 1^-$.
As computed in Example~\ref{Ex: characters on weight space}, all powers $x^k$ of the cyclotomic character are \emph{not} in the rim of the weight space as soon as $r > p^{-1}$.

We put $\Spc^{\geq r}: = \mathrm{wt}^{-1}(\calW^{\geq r}).$

The following question was asked by Coleman and Mazur \cite{coleman-mazur}, and later elaborated by Buzzard and Kilford \cite{buzzard-kilford}.\footnote{In the recent preprint of the first two authors and R. Liu \cite{liu-wan-xiao}, we proved many cases of this conjecture.}
\begin{conjecture}
\label{Conj:CM-conj}
When $r$ is sufficiently close to $1$, the following statements hold.
\begin{enumerate}
\item
The space $\Spc^{\geq r}$ is a disjoint union of (countably infinitely many) connected components $X_1, X_2, \dots$ such that the weight map $\wt: X_n \to \calW^{\geq r}$ is finite and flat for each $n$.
\item
There exist nonnegative rational numbers $\lambda_1, \lambda_2, \dots \in \QQ$ in non-decreasing order and approaching to infinity such that, for each $i$ and each point $z \in X_n$, we have
\[
|a_p(z)| = |\mathrm{wt}(z)|^{(p-1)\lambda_n}.
\]
\item
The sequence $\lambda_1, \lambda_2, \dots$ is a disjoint union of finitely many arithmetic progressions, counted with multiplicity (at least when the indices are large enough).
\end{enumerate}
\end{conjecture}

Clearly Conjecture~\ref{Conj:CM-conj} implies Conjecture~\ref{Conj:weak-buzzard-kilford} by specializing to classical weights using Coleman's classicality result \cite{coleman1,coleman2}.

\begin{remark}
\label{R:remark after the conjecture}
Let us give a few evidences and remarks on Conjecture~\ref{Conj:CM-conj} (as well as Conjecture~\ref{Conj:weak-buzzard-kilford}).
\begin{enumerate}

\item
The novelty of our formulation lies in emphasizing statement (3) of Conjecture~\ref{Conj:CM-conj} as part of the general picture.
In fact, the aim of this paper is to give strong evidence to support this expectation; see in particular, Corollary~\ref{C:precise NP computation} and Theorem~\ref{T:main theorem each residual}(3).

\item Similar properties near the center of the weight space are expected to be false; we refer to \cite{buzzard-calegari2, buzzard-slope, clay, loeffler} for more discussions.\footnote{Very recently, Bergdall and Pollack \cite{bergdall-pollack} gave an interesting conjecture regarding the slopes of eigencurve at the center of the weight space.}  But see also Remarks~\ref{R:heuristic gouvea}.

\item
One can reformulate this conjecture for eigencurves instead of spectral curves; the two statements would be essentially equivalent.

\item When $p=2,3$ and the modular curve is taken to be $X_0(p)$, Conjecture~\ref{Conj:CM-conj} is proved using direct computations by Buzzard--Kilford \cite{buzzard-kilford} and Roe \cite{roe}, extending the thesis of Emerton \cite{emerton}.

\item For $p=5,7$, the weaker version Conjecture~\ref{Conj:weak-buzzard-kilford} was verified in some cases by Kilford and McMurty \cite{kilford, kilford-mcmurdy}.

\item In an analogous situation where the eigencurve associated to Artin--Schreier--Witt tower of curves is considered, the analogue of Conjecture~\ref{Conj:CM-conj}, in fact over the entire weight space,\footnote{The fact that the analogous statements hold over the entire weight space means that the situation is largely simplified; the method will probably not translate directly to the Coleman--Mazur eigencurve case.} is proved by  Davis and the first two authors \cite{davis-wan-xiao}.
Our argument in Section~\ref{Section:improve Hodge polygon} shares some similarities with this approach and is in part inspired by it.
\end{enumerate}
\end{remark}

\begin{remark}
We give our most optimistic expectation of the numerics in Conjecture~\ref{Conj:CM-conj}.
Suppose $ p \geq 3$ for simplicity. First, we expect Conjecture~\ref{Conj:CM-conj} to hold for $r = p^{-1/(p-1)}$ (i.e. the radius for finite characters of conductor $p^2$).\footnote{It is possible that Conjecture~\ref{Conj:CM-conj} holds for even smaller $r$, e.g. $r < p^{-1}$; but we do not have strong evidence either supporting or against this.}
Moreover, we hope to make a guess about the sequence $\lambda_1, \lambda_2, \dots $ in Conjecture~\ref{Conj:CM-conj}.
Assume that the tame level structure is \emph{neat}.
Fix a connected component of the weight disk and fix a finite character $\psi_2$ of conductor $p^2$ so that the character $x \psi_2$ lies in that weight disk.

For $i=0, \dots, \frac{p-3}2$,
consider the action of $U_p$ on the space of cusp forms $S_2(p^2; \psi_2 \omega^{-2i})$ whose tame level is as given and the level at $p$ is $\Gamma_0(p^2)$ with nebentypus character $\psi_2\omega^{-2i}$.
The dimension of such space is denoted by $t$  (which does not depend on $i$).  Let $\alpha_1^{(i)}, \dots, \alpha_t^{(i)}$ denote the $p$-adic valuations of the corresponding $U_p$-eigenvalues, counted with multiplicity.

Let $d$ denote the number of cusps of the modular curve with only the tame level, or equivalently the dimension of the weight $2$ Eisenstein series for the tame level.

Then the sequence $\lambda_1,\lambda_2, \dots$ is expected to be the union (rearranged into the non-decreasing order) of exactly the following list of numbers:
\begin{itemize}
\item
the
numbers $1, 2, 3, \dots$ with multiplicity $d$, and
\item
for $i=0, \dots, \frac{p-1}{2}$ and $r = 1, \dots, t$, the numbers
\[
\alpha_r^{(i)} + i,\ \alpha_r^{(i)}+i + \tfrac{p-1}{2},\ \alpha_r^{(i)}+i + (p-1),\ \dots.
\]
\end{itemize}
The former part should be considered as ``contributions from the Eisenstein series" although the overconvergent modular forms are cuspidal; and the latter part is the ``contributions from the cuspidal part", which is a union of arithmetic progressions with common difference $\frac{p-1}{2}$.  (The number $\frac{p-1}2$ comes from the cyclic repetition of powers of the Teichm\"uller character.)
Our guess is motivated by the main theorems of this paper and some computation of Kilford and McMurty \cite{kilford, kilford-mcmurdy}.
\end{remark}

\section{Automorphic forms for a definite quaternion algebra}
\label{Section:automorphic forms}
One of the major technical difficulties, among others, is the poor understanding of the geometry of the modular curves, in explicit coordinates.
To bypass this difficulty, we consider the eigencurve for a definite quaternion algebra; then a $p$-adic family version of Jacquet--Langlands correspondence \cite{chenevier} allows us to recover a big part of Conjecture~\ref{Conj:CM-conj}, from the corresponding statements for the quaternion algebra.
We now recall the definition of the quaternionic eigencurves following \cite{buzzard2,buzzard}.

\subsection{Setup}
\label{S:setup for D}
Let $\AAA_f$ denote the finite adeles of $\QQ$ and $\AAA_f^{(p)}$ its prime-to-$p$ components.
Let $D$ be a definite quaternion algebra over $\QQ$ which splits at $p$; in other words,
$D \otimes_\QQ \RR$ is isomorphic to the Hamiltonian quaternion and
$D \otimes_{\QQ} \Qp \simeq \rmM_2(\Qp)$.
Put $D_f : = D \otimes_\QQ \AAA_f$.
Let $\calS$ be a finite set of primes including $p$ and all primes at which $D$ ramifies.
For each prime $l \neq p$, we fix an open compact subgroup of $U_l$ of  $(D \otimes_\QQ \QQ_l)^\times$.
For $l \notin \calS$, we fix an isomorphism $D \otimes_\QQ \QQ_l \simeq \rmM_2(\QQ_l)$ and require that $U_l \simeq \GL_2(\ZZ_l)$ under this identification.
We fix a positive integer $m \in \NN$ and consider the Iwahori subgroup
\[
U_0(p^m) =
\begin{pmatrix}
\ZZ_p^\times & \Zp \\ p^m \Zp & \Zp^\times
\end{pmatrix} \subset \GL_2(\Qp) \simeq (D \otimes_\QQ \QQ_p)^\times.
\]
We will later need the monoid
\[
\Sigma_0(p^m): =\Big\{\gamma =\big(
\begin{smallmatrix}
a&b\\c&d
\end{smallmatrix} \big) \in \rmM_2(\ZZ_p) \; \Big| \;
p^m|c, \ p\nmid d,\, \det(\gamma) \neq 0 \Big\}.
\]
Finally, we write $U = \prod_{l \neq p} U_l \times U_0(p^m)$ for the product, as an open compact subgroup of $D_f^\times$.
We occasionally use $U_1$ to denote $\prod_{l \neq p} U_l \times \big(\begin{smallmatrix} \Zp^\times & \Zp\\ p^m\ZZ_p & 1+p^m \Zp \end{smallmatrix}\big)$.

We further assume that $U$ is taken sufficiently small so that
(see \cite[Section~4]{buzzard2})
\begin{equation}
\label{E:buzzard condition}
\textrm{for any }x \in D_f^\times, \textrm{ we have }x^{-1} D^\times x \cap U = \{1\}.
\end{equation}

We fix a finite extension $E$ of $\Qp$ as the coefficient field, which we will enlarge as needed in the argument.
Let $\calO$ denote the valuation ring of $E$ and $\varpi$ an uniformizer. Write $\FF = \calO / (\varpi)$ for the residue field.
Let $v(\cdot)$ denote the valuation on $E$ normalized so that $v(p) = 1$.

We write $\calA^\circ := \calO \langle z \rangle$ and $\calA = \calA^\circ[\frac1p]$ for the Tate algebras.

Put $r_m = p^{-1/p^{m-1}(p-1)}$ if $p>2$ and $r_m = 2^{-1/2^{m-2}}$ if $p=2$.
Let $\calW^{< r_m}$ denote the open disks of $\calW$ where the $T$-coordinate has absolute value $< r_m$.
Let $\Max(A)$ be an affinoid space over $\calW^{< r_m}$.  (Typical examples of $\Max(A)$ we consider are either a subdomain or a point.)  Let $\kappa: \Gamma \to A^\times$ denote the universal character.
Then $\kappa$ extends to a continuous character
\begin{align}
\label{E:extend kappa A}
\kappa: (\ZZ_p + p^m \calA^\circ)^\times = \ZZ_p^\times \cdot (1+p^m&\calA^\circ)^\times \longrightarrow (A \widehat \otimes \calA^\circ)^\times
\\
\nonumber
a\cdot x &\longmapsto \kappa(a) \cdot \kappa(\exp(2p))^{(\log x)/{2p}}.
\end{align}
One checks easily that the condition $|\kappa(\exp(2p))-1|< r_m$ ensures the convergence and the independence of the factorization $a\cdot x$. See e.g. \cite[Section~2.1]{pilloni} for a more optimal convergence condition.

\subsection{Overconvergent automorphic forms}
\label{S:quaternionic forms}
Consider the right action of $\Sigma_0(p^m)$ on $A\widehat\otimes\calA$ given by
\begin{equation}
\label{E:right action}
\textrm{for } \gamma = \big(
\begin{smallmatrix}
a&b\\c&d
\end{smallmatrix} \big) \in \Sigma_0(p^m) \textrm{ and }h(z) \in A \widehat \otimes \calA, \quad
(h||_{\kappa}\gamma)(z): = \frac{\kappa(cz+d)}{cz+d} h\big( \frac{az+b}{cz+d}\big).\footnote{Our weight normalization is different from \cite{buzzard2, buzzard, jacobs} and most of the literature by using $cz+d$ in the denominator as opposed to $(cz+d)^2$; we will see a small benefit of our choice later in Proposition~\ref{P:generating series}.}
\end{equation}
Note that it is crucial that $p^m |c$ and $d \in \ZZ_p^\times$ so that $\kappa(cz+d)$ and $(cz+d)^{-1}$ make sense.

We define
the space of \emph{overconvergent automorphic forms of weight $\kappa$ and level $U$} to be
\[
S^{D,\dagger}(U; \kappa): = \Big
\{
\varphi: D^\times_f \to A \widehat \otimes \calA\;\Big|\;
\varphi(\delta gu) = \varphi(g)||_{\kappa}u_p, \textrm{ for any }\delta \in D^\times, g \in D_f^\times, u \in U
\Big\},
\]
where $u_p$ is the $p$-component of $u$.

\begin{example}
\label{Ex:classical weight}
When $\kappa = x^k\psi_{m'}: \Gamma \to \QQ_p(\zeta_{p^{m'-1}})^\times$ is the continuous character considered in Example~\ref{Ex: characters on weight space},
we can take the definition above for $A = E \supset \Qp(\zeta_{p^{m'-1}})$ corresponding to the point $\kappa$ on $\calW$, which lies
in $\calW^{\leq r_m}$ if $m' \leq m$.
In this case, the right action is given by
\begin{equation}
\label{E:chi action}
(h||_\kappa\gamma)(z)= (cz+d)^{k-1} \psi_{m'}(d) h\big(\frac{az+b}{cz+d} \big).
\end{equation}
The space $S^{D,\dagger}(U; \kappa) = S^{D,\dagger}_{k+1}(U;\psi_{m'})$ is the space of \emph{overconvergent automorphic forms of weight $k+1$, nebentypus character $\psi_{m'}$, and level $U$}.

Moreover, when $k \geq 1$ is a positive integer, we observe that the subspace $L_{k-1}$ of $\calA$ consisting of polynomials in $z$ with degree $\leq k-1$ is stable under the action \eqref{E:chi action}; so we can define the space of \emph{classical automorphic forms of weight $k+1$, character $\psi_{m'}$, and level $U$} to be the subspace $S^D_{k+1}(U; \psi_{m'})$ of $S^{D ,\dagger}_{k+1}(U; \psi_{m'})$ consisting of functions $\varphi$ with values in $L_{k-1}$.
In particular, when $k=1$,
\begin{align}
\label{E:S2 classical}
S_2^D(U; \psi_m): = \big\{ \varphi: D^\times_f \to E \, \big| \, \varphi(\delta gu) =  \psi_m(d) & \varphi(g) \textrm{ for any }\delta \in D^\times, g \in D^\times_f,
\\
\nonumber
&\textrm{and } u\in U \textrm{ with }u_p = \big(\begin{smallmatrix}
a&b\\c&d
\end{smallmatrix} \big) \big\}.\footnote{By Jacquet--Langlands, constant function on $D^\times \backslash D^\times_f/U$ corresponds to weight two modular forms.}
\end{align}
We occasionally write $S_2^D(U; \psi_m; \calO)$ for the subspace of functions that take values in $\calO$ (as opposed to $E$).
\end{example}

\subsection{Hecke actions}
\label{S:Hecke action}
The space $S^{D,\dagger}(U; \kappa)$ carries actions of Hecke operators, which preserves the subspace of classical automorphic forms $S^D_{k+1}(U; \psi_m)$ when $\kappa = x^k\psi_m$ is given as in Example~\ref{Ex:classical weight}.

Let $l$ be a prime not in $\calS$; then $U_l \simeq \GL_2(\ZZ_l)$.
We write $U_l \big( \begin{smallmatrix}
l&0\\0&1
\end{smallmatrix}\big) U_l = \coprod_{i=0}^{l} U_l w_i$, with $w_i = \big( \begin{smallmatrix}
l&0\\i&1
\end{smallmatrix}\big)$ for $i =0, \dots, l-1$ and $w_l = \big( \begin{smallmatrix}
1&0\\0&l
\end{smallmatrix}\big)$, viewed as elements in $\GL_2(\QQ_l) \simeq D \otimes_\QQ \QQ_l$.
We define the action of the operator $T_l$ on $S^{D,\dagger}(U; \kappa)$ by
\[
T_l(\varphi) = \sum_{i =0}^l \varphi|_{\kappa}w_i, \quad \textrm{with }(\varphi|_{\kappa}w_i)(g): = \varphi(gw_i^{-1}).\footnote{This looks slightly different from \eqref{E:defn of U_p} below because $||_\kappa w_i$ is trivial as $w_i$ is not in the $p$-component.}
\]

Similarly, we write (note $m \geq 1$)
\begin{equation}
\label{E:Up operator cosets}
U_0(p^m)\big(\begin{smallmatrix}
p&0\\0&1
\end{smallmatrix}\big) U_0(p^m) = \coprod_{i=0}^{p-1}
U_0(p^m) v_i, \quad \textrm{with } v_i = \big(\begin{smallmatrix}
p&0\\ip^{m} &1
\end{smallmatrix}\big).
\end{equation}
Then the action of the operator $U_p$ on $S^{D,\dagger}(U; \kappa)$ is defined to be
\begin{equation}
\label{E:defn of U_p}
U_p(\varphi) = \sum_{i =0}^{p-1} \varphi|_{\kappa}v_i, \quad \textrm{with }(\varphi|_{\kappa}v_i)(g): = \varphi(gv_i^{-1})||_{\kappa}v_i.
\end{equation}
We point out that the definition of $U_p$- and $T_l$-operators do not depend on the choices of the double coset representatives $w_i$ and $v_i$.  But our choices may ease the computation.

These $U_p$- and $T_l$-operators are viewed as acting on the space on the left (although the expression seems to suggest a right action); they are pairwise commutative.

\begin{notation}
If an (overconvergent) automorphic form $\varphi$ is a (generalized) eigenvector for the $U_p$-operator, we call the $p$-adic valuation of its (generalized) $U_p$-eigenvalue the \emph{$U_p$-slope} or simply the \emph{slope} of $\varphi$.
By \emph{$U_p$-slopes} on a space of (overconvergent) automorphic forms, we mean the set of slopes of all generalized $U_p$-eigenforms in this space, counted with multiplicity.
\end{notation}

\subsection{Classicality of automorphic forms}
The relation between the classical and the overconvergent automorphic forms in weight $k+1 \geq 2$ can be summarized by the following exact sequence
\[
0 \to S_{k+1}^D(U; \psi_m) \to S_{k+1}^{D, \dagger}(U; \psi_m) \xrightarrow{(\frac{d}{dz})^k}  S_{1-k}^{D, \dagger}(U; \psi_m) \to 0,
\]
where the first map is the natural embedding and the second map is given by
\[
 \big(  \big(\frac{d}{dz} \big)^k (\varphi) \big) (g) : =
 \big(\frac{d}{dz} \big)^k \big( \varphi(g)\big).
\]
One checks that $ (\frac{d}{dz} )^k \circ U_p= p^k \cdot U_p \circ (\frac{d}{dz} )^k$ (see \cite[\S 7]{buzzard2}).
As a corollary, all $U_p$-eigenforms of $S_{k+1}^{D, \dagger}(U; \psi_m)$ with slope strictly less than $k$  are classical.
It is also well known that the $U_p$-slopes on $S_{k+1}^D(U; \psi_m)$ are always less than or equal to $k$  by the admissibility of the associated Galois representation.
It follows that the $U_p$-slopes on $S_{k+1}^D(U; \psi_m)$ are exactly the smallest $\dim S_{k+1}^D(U; \psi_m)$\footnote{This number can be expressed in a simple way as in Corollary~\ref{C:dimension formula}.} numbers (counted with multiplicity) in the set of $U_p$-slopes on $S_{k+1}^{D, \dagger}(U; \psi_m)$.

\subsection{Jacquet--Langlands correspondence}
\label{S:jacquet langlands}
We recall a very special case of the classical Jacquet--Langlands correspondence, which was used in the introduction. Let $N$ be a positive integer coprime to $p$.
Assume that there exists a prime number $\ell$ such that $\ell ||N$.
Let $D_{\ell\infty}$ denote the definite quaternion algebra over $\QQ$ which ramifies at exactly $\ell$ and $\infty$.
If we take the level structure so that $\calS$ is the set of prime factors of $p\ell N$,  $U_\ell$ is the maximal open compact subgroup of $(D_{\ell\infty}\otimes\QQ_\ell)^\times$, and $U_q =  \big( \begin{smallmatrix}
\ZZ_q^\times & \ZZ_q \\ N\ZZ_q & \ZZ_q^\times
\end{smallmatrix} \big) \subset \GL_2(\QQ_q) \simeq (D_{\ell\infty} \otimes \QQ_q)^\times$ for a prime $q |N$ but $q \neq \ell , p$, then the Jacquet--Langlands correspondence says that there exists an isomorphism of modules of $U_p$- and all $T_q$-operators for $q \nmid Np$:
\begin{equation}
\label{E:Jaquet-Langlands classical}
S_{k+1}(\Gamma_0(p^mN); \psi_m)^{\ell\new} \cong S_{k+1}^{D_{\ell\infty}}(U; \psi_m)
\end{equation}
for all weights $k+1 \geq 2$.
This allows us to translate our results about automorphic forms on definite quaternion algebras to results about modular forms.
One can certainly make variants of this, but we do not further discuss.

\subsection{Eigencurve for $D$}
It is clear that $S^{D, \dagger}(U;\kappa)$ satisfies Buzzard's property (Pr) (see \ref{S:CM eigencurve}), by an argument similar to \cite[\S 10]{buzzard} or imitate Lemma~\ref{L:explicit space of automorphic forms}.
The action of  the $U_p$-operator on $S^{D, \dagger}(U; \kappa)$ is nuclear by
\cite[Lemma~12.2]{buzzard}.
So the construction in Subsection~\ref{S:CM eigencurve} applies with $S = S^{D, \dagger}(U; \kappa)$ to give a spectral curve over $\Max(A)$.
The construction is clearly functorial in $A$ and hence defines a spectral curve $\Spc_D$ over $\calW^{<r_m}$.
As explained in \cite[Section~13]{buzzard}, the construction for different $m$ also glues over small weight disks and hence gives rise to a spectral curve $\Spc_D$ over the entire weight space $\calW$.

The
Jacquet--Langlands correspondence above can be made into $p$-adic families.
By \cite{chenevier}, there is a closed immersion $\Spc_D^\mathrm{red}\hookrightarrow \Spc^\mathrm{red}$, where the superscript means to take the reduced subscheme structure.\footnote{Rigorously speaking, \cite{chenevier} proves the result for eigencurves; but the spectral curves, when taking the reduced scheme structure, are exactly the images of the eigencurves after forgetting the tame Hecke actions.}
Therefore, it is natural to expect that Conjecture~\ref{Conj:CM-conj} holds for $\Spc_D$ in place of $\Spc$.
Conversely, knowing Conjecture~\ref{Conj:CM-conj} for $\Spc_D$, it is quite possible to infer a lot of information regarding $\Spc$ via the comparison \cite{chenevier}.

\section{Explicit computation of the $U_p$-operators}
\label{Section:computation of Up}

We now make the first attempt to prove certain weak version of Conjectures~\ref{Conj:weak-buzzard-kilford} and \ref{Conj:CM-conj}, ending with a proof of Theorem~\ref{T:theorem A}.
To our best knowledge, the only known approach to any form of these conjectures
is via ``brutal force" computation, that is to compute directly the characteristic power series of the $U_p$-operator to the extent that one can determine its slopes.
Our approach is derived from a computation made by Jacobs \cite{jacobs} of the infinite matrix for $U_p$ in terms of concrete numbers.
The novelty of our improvement is to make ``brutal but formal computation" as opposed to using numbers.
We include his example in the next section with some simplification.  It serves as a toy model of our computation presented in this section.

\begin{notation}
\label{N:representatives}
We decompose $D^\times_f$ into (a disjoint union of) double cosets $\coprod_{i=0}^{t-1} D^\times \gamma_i U$, for some elements $\gamma_0,\gamma_1, \dots, \gamma_{t-1} \in D^\times_f$.
By our smallness hypothesis on $U$ in Subsection~\ref{S:setup for D}, the natural map $D^\times \times U \to D^\times \gamma_i U$ for each $i$ sending $(\delta,u)$ to $\delta\gamma_iu$ is bijective.  We say that the double coset decomposition above is \emph{honest}.

Since the  norm map $\Nm: D^\times \to \QQ^\times_{>0}$ is surjective, we may modify  the representatives $\gamma_i$ so that $\Nm(\gamma_i) \in \widehat \ZZ^\times$.  Moreover, since $\Nm(U_0(p^m)) = \ZZ_p^\times$, we can further modify the $p$-component of each $\gamma_i$ so that its norm is $1$.  Finally, using the fact that $(D^\times)^{\Nm=1} $ is dense in $(D\otimes_\QQ \Qp)^{\times, \Nm=1}$, we may assume that the $p$-component of each $\gamma_i$ is trivial, still keeping the property that $\Nm(\gamma_i) \in \widehat \ZZ^\times$.
\end{notation}

Let $\Max(A)$ be an affinoid space over $\calW^{ < r_m}$ and let $\kappa: \Gamma \to A^\times$ be the universal character.

\begin{lemma}
\label{L:explicit space of automorphic forms}
We have an $A$-linear isomorphism of Banach spaces
\[
\xymatrix@R=0pt{
S^{D,\dagger}(U; \kappa) \ar[r]^\cong &
\oplus_{i=0}^{t-1}A \widehat \otimes \calA\\
\varphi \ar@{|->}[r] & \big(\varphi(\gamma_i) \big)_{i = 0, \dots, t-1}.
}
\]
\end{lemma}
\begin{proof}
This is clear as the function $\varphi$ is uniquely determined by its value at the chosen representatives $\gamma_i$.
There is no further restriction on the value of $\varphi(\gamma_i)$ because the double coset decomposition in Notation~\ref{N:representatives} is honest.
\end{proof}

\begin{corollary}
\label{C:dimension formula}
We have $\dim S_{k+1}^D(U; \psi_m) = kt$, for the number $t$ in Notation~\ref{N:representatives}.
\end{corollary}

\begin{proposition}
\label{P:explicit Up}
In terms of the explicit description of the space of overconvergent automorphic forms, the $U_p$- and $T_l$- (for $l \notin \calS$) operators can be described by the following commutative diagram.
\[
\xymatrix@C=80pt{
S^{D,\dagger}(U; \kappa) \ar[r]^{\varphi \mapsto (\varphi(\gamma_i))} \ar[d]_{\begin{tiny}\begin{split}\varphi \mapsto U_p\varphi\\ \varphi \mapsto T_l \varphi\end{split}\end{tiny}} &
\oplus_{i=0}^{t-1} A \widehat \otimes\calA \ar[d]^{\begin{tiny}\begin{split}\textrm{Map of}\\ \textrm{interest} \end{split}\end{tiny}\quad}_{\begin{tiny}\begin{split}\gothU_p\\ \gothT_l\end{split}\end{tiny}}
\\
S^{D,\dagger}(U; \kappa) \ar[r]^{\varphi \mapsto (\varphi(\gamma_i))} &
\oplus_{i=0}^{t-1} A \widehat \otimes \calA.
}
\]
Here the right vertical arrow $\gothU_p$ (resp. $\gothT_l$) is given by a matrix with the following description.
\begin{itemize}
\item[(1)] The entries of $\gothU_p$ (resp. $\gothT_l$) are sums of operators of the form $||_\kappa \delta_p$, where $\delta_p$ is the $p$-component of a \emph{global} element $\delta \in D^\times$ \emph{of norm $p$ (resp. norm $l$)}.
\item[(2)] There are exactly $p$ (resp. $l+1$) such operators appearing in each row and each column  of $\gothU_p$ (resp. $\gothT_l$).
\item[(3)] We have
$
\delta_p \in \big(\begin{smallmatrix} p\ZZ_p& \ZZ_p\\p^{m}\ZZ_p&\ZZ_p^\times \end{smallmatrix}\big)
$ (resp. $
\delta_p \in U_0(p^m) = \big( \begin{smallmatrix} \ZZ_p^\times& \ZZ_p\\p^{m}\ZZ_p&\ZZ_p^\times \end{smallmatrix} \big)
$).
\end{itemize}
\end{proposition}
\begin{proof}
We only prove this for the $U_p$-operator and the proof for the $T_l$-operator ($l \notin \calS$) is similar.  For each $\gamma_i$, we have
\[
(U_p \varphi)(\gamma_i) = \sum_{j=0}^{p-1}
\varphi(\gamma_i v_j^{-1})||_\kappa v_j.
\]
Now we can write each $\gamma_iv_j^{-1}$ \emph{uniquely} as $\delta_{i,j}^{-1} \gamma_{\lambda_{i,j}} u_{i,j}$ for $\delta_{i,j} \in D^\times$, $\lambda_{i,j} \in \{0, \dots, t-1\}$, and $u_{i,j} \in U$.
Then we have
\[
(U_p \varphi)(\gamma_i) = \sum_{j=0}^{p-1}
\varphi(\delta_{i,j}^{-1} \gamma_{\lambda_{i,j}} u_{i,j})||_\kappa v_j = \sum_{j=0}^{p-1}
\varphi( \gamma_{\lambda_{i,j}})||_\kappa  (u_{i,j,p}v_j),
\]
where $u_{i,j,p}$ is the $p$-component of $u_{i,j}$.
Substitute back in $u_{i,j}v_j = \gamma_{\lambda_{i,j}}^{-1} \delta_{i,j} \gamma_i$ and note the fact that both $\gamma_i$ and $\gamma_{\lambda_{i,j}}$ have trivial $p$-component by our choice in Notation~\ref{N:representatives}.  We have
\[
(U_p\varphi)(\gamma_i) =
\sum_{j=0}^{p-1} \varphi(\gamma_{\lambda_{i,j}})||_\kappa \delta_{i,j,p},
\]
where $\delta_{i,j,p}$ is the $p$-component of the \emph{global element} $\delta_{i,j} \in D^\times$.
We now check the description of each $\delta_{i,j}$:
\[
\delta_{i,j} = \gamma_{\lambda_{i,j}} u_{i,j} v_j \gamma_i^{-1}  \in \gamma_{\lambda_{i,j}} U\big(\begin{smallmatrix} p&0\\0&1 \end{smallmatrix}\big)U \gamma_i^{-1}.\]
From this, we see that the $p$-component of $\delta_{i,j}$ lies in $\big(\begin{smallmatrix} p\ZZ_p& \ZZ_p\\p^{m}\ZZ_p&\ZZ_p^\times \end{smallmatrix}\big)$.
Moreover, the norm of
$\gamma_{\lambda_{i,j}} U\big(\begin{smallmatrix} p&0\\0&1 \end{smallmatrix}\big) U \gamma_i^{-1}$  lands in $p \widehat \ZZ^\times$, because our choice of the representatives satisfies $\Nm(\gamma_{i}) \in \widehat \ZZ^\times$  by Notation~\ref{N:representatives}.
Therefore, $\Nm(\delta_{i,j}) \in \QQ^\times_{>0} \cap p\widehat \ZZ^\times = \{p\}$.  This concludes the proof of the proposition.
\end{proof}

\subsection{Infinite matrices and generating functions}
For an infinite matrix (where the row and column indices start with $0$ as opposed to $1$)

\begin{equation}
\label{E:infinite matrix}
M = \begin{pmatrix}
m_{0,0} & m_{0,1} & m_{0,2} &\cdots\\
m_{1,0} & m_{1,1} & m_{1,2} &\cdots\\
m_{2,0} & m_{2,1} & m_{2,2} &\cdots\\
\vdots & \vdots & \vdots & \ddots
\end{pmatrix}
\end{equation}
with coefficients in an affinoid $E$-algebra $A$, we consider the following formal power series:
\[
H_M(x,y) = \sum_{i,j \in \ZZ_{\geq 0}} m_{i,j} x^i y^j \in A\llbracket x,y\rrbracket.
\]
It is called the \emph{generating series} of the matrix $M$.
When $M$ is the matrix for an operator $T$ acting on the Tate algebra $A \widehat \otimes \calA = A\langle z\rangle$ over $A$ with respect to the basis $1, z, z^2, \dots$, we call $H_M(x,y)$ the \emph{generating series} of $T$.

For $u \in E$, we write $\Diag(u)$ for the infinite diagonal matrix with diagonal elements $1, u, u^2, \dots$.
Then we have
\[
H_{\Diag(u)M\Diag(v)} (x,y)= H_M(u x,v y).
\]
For $t \in \NN$, we write $\Diag(u; t)$ for the infinite diagonal matrix with diagonal elements $1, \dots, 1, u, \dots, u, u^2, \dots$ where each number appears repeatedly $t$ times.

The following key calculation is due to Jacobs \cite[Proposition~2.6]{jacobs}.

\begin{proposition}
\label{P:generating series}
Let $\kappa: \Gamma \to A^\times$ be the universal character  for an affinoid space $\Max(A)$ over $\calW^{< r_m}$.  Let $\big(\begin{smallmatrix}a&b\\c&d\end{smallmatrix} \big)$ be a matrix in $\Sigma_0(p^m)$.
The generating series of the operator $||_{\kappa}\big(\begin{smallmatrix}a&b\\c&d\end{smallmatrix} \big)$ acting on $A \widehat \otimes \calA$ (with respect to the basis $1, z, z^2, \dots$) is given by
\[
\frac{\kappa(c x+d)}{c x+d-a xy-b y}.
\footnote{Comparing to the convetion in \cite{jacobs}, we loose an extra factor of $cx+d$ in the denominator due to our normalization \eqref{E:right action}.  There is no real improvement in our formula except that it looks shorter.}
\]
Here we point out that, although the operator $||_{\kappa}\big(\begin{smallmatrix}a&b\\c&d\end{smallmatrix} \big)$ when viewed as the action of the monoid $\Sigma_0(p^m)$ is a right action, we only use one particular operator and will not discuss the composition; so we still use the column vector convention (pretending it as a left operator).
\end{proposition}
\begin{proof}
This is straightforward.
By definition,
\begin{align*}
H_{||_{\kappa}\big(\begin{smallmatrix}a&b\\c&d\end{smallmatrix} \big)}(x,y)  &=
\sum_{i \in \ZZ_{\geq 0}}
y^i \frac{\kappa(c x+d)}{c x+d}\cdot \big(\frac{a x+b}{c x+d}\big)^i\\
&=
\frac{\kappa(c x+d)}{c x+d}\cdot \frac1{1-y \cdot \frac{a x+b}{c x+d}}
=\frac{\kappa(c x+d)}{c x+d-a xy-b y}.\qedhere
\end{align*}
\end{proof}

Combining Proposition~\ref{P:generating series} with Proposition~\ref{P:explicit Up}, we can give a good description of the infinite matrices for $\gothU_p$ and $\gothT_l$ (for $l \notin \calS$).


\subsection{Hodge polygon and Newton polygon of a matrix}
\label{S:Hodge v.s. Newton}
Before proceeding, we remind the readers some basic facts about $p$-adic analysis. We will use them later freely without referencing back here. Let $M \in \rmM_n(E)$ be an $n\times n$-matrix.
\begin{enumerate}
\item
The \emph{Newton polygon} of $M$ is the convex polygon starting at $(0,0)$ whose slopes are exactly the $p$-adic valuations of the eigenvalues of $M$, counted with multiplicity.
\item

The \emph{Hodge polygon} of $M$ is the convex hull of the vertices
\[
 \big( i,\  \textrm{the  minimal $p$-adic valuation of the determinants of all $i\times i$-minors} \big).
\]

\item
The Hodge polygon is  invariant when conjugating $M$ by elements in $\GL_n(\calO)$; the Newton polygon is  invariant when conjugating $M$ by elements in $\GL_n(E)$.

\item
If the slopes of the Hodge polygon of $M$ are $a_1\leq  \dots\leq a_n$, then there exist matrices $A, B \in \GL_n(\calO)$ such that $AMB$ is a diagonal matrix whose diagonal elements have valuation exactly $a_1, \dots, a_n$.
Conversely, if such $A$ and $B$ exist, the Hodge polygon of $M$ has the described slopes.

\item
If the slopes of the Hodge polygon of $M$ are $a_1\leq  \dots\leq a_n$, then there exists a matrix $A \in  \GL_n(\calO)$ such that the valuations of all entries in the $i$-th row of $AMA^{-1}$ are at least $ a_i$ for all $i$.
Conversely, when such $A$ exists, the Hodge polygon of $M$ lies above the polygon with slopes $a_1, \dots, a_n$.

\item
When $\calO^{\oplus n}$ can be written as $V \oplus V'$ for two $M$-stable $\calO$-submodules.
Then the set of Newton slopes for $M$ is the union of the sets of Newton slopes of $M$ acting on $V$ and $V'$, counted with multiplicity.  The same holds for Hodge slopes.

\item
It is always true that the Newton polygon lies above the Hodge polygon.  This also holds for an  infinite matrix associated to a nuclear operator.
\end{enumerate}

We now prove Theorem~\ref{T:theorem A} from the introduction (through the Jacquet--Langlands correspondence \eqref{E:Jaquet-Langlands classical}):

\begin{theorem}
\label{T:weak Hodge polygon}
Let $\psi_m$ be a finite character of $\Zp^\times$ of conductor $p^m$ with the same $m$ that defines the level structure $U$.\footnote{Again, we allow $\psi$ to be trivial, in which case $m=1$ if $p>2$ and $m=2$ if $p=2$.}
Recall that  $\dim S_2^D(U; \psi_m)=t$.
Then the Newton polygon for the slopes of $U_p$ acting on $S^{D,\dagger}(U, x\langle x\rangle^k\psi_m)$ for $k \in \ZZ_p$ lies above the polygon with vertices
\begin{equation}
\label{E:trivial hodge bound}
(0,0), (t,0), (2t, t), \dots, (nt, \tfrac{n(n-1)}2 t), \dots.
\end{equation}
\end{theorem}
\begin{proof}
By Lemma~\ref{L:explicit space of automorphic forms} and Proposition~\ref{P:explicit Up} (in our case $A = E = \QQ_p(\zeta_{p^{m-1}})$), it suffices to understand the matrix for the operator $\gothU_p$.
We first give $\oplus_{i=0}^{t-1} \calA$ a basis:
\[
1_0, z_0, z_0^2, \dots, 1_1, z_1, z^2_1, \dots, 1_{t-1}, z_{t-1}, \dots,
\]
 where the subscripts indicate which copy of $\calA$ the element comes from. Then the matrix for $\gothU_p$ is a $t\times t$-block matrix such that each block is an infinite matrix.
By Proposition~\ref{P:generating series}, the generating series of each block is the sum of power series of the form
\[
\frac{d\psi_m(d) \langle d\rangle ^k(1+ \frac cd x)^{k+1}}{cx+d-axy-by}, \quad \textrm{with } p|a\textrm{ and }p^m|c \textrm{ by Proposition~\ref{P:explicit Up}}.
\]
When $k \in \ZZ_p$, the expression above lands in $\calO\llbracket  p^mx, pxy,y\rrbracket \subseteq \calO\llbracket px, y\rrbracket$. In particular, the $i$th row of the corresponding infinite matrix is divisible by $p^i$.

We can then rewrite the matrix of $\gothU_p$ under the following basis of $\oplus_{i=0}^{t-1} \calA$:
\[
1_0,1_1, \dots, 1_{t-1}, z_0,\dots, z_{t-1}, z_0^2, \dots.
\]
Then the matrix of $\gothU_p$ becomes an infinite block matrix, where each block is $t\times t$.
Moreover, the discussion above implies that the $i$th block row    is entirely divisible by $p^i$.
In other words, the Hodge polygon of this matrix lies above the polygon with vertices given by \eqref{E:trivial hodge bound}.
So the Newton polygon of $\gothU_p$ also lies above it.
\end{proof}

\begin{remark}
\label{R:improve Hodge bound classical forms}
We discuss how one can improve the lower bound of the Newton polygon of the $U_p$-action on  the classical automorphic forms $S_{k+1}^D(U; \psi)$ when $\psi$ has conductor $p$ and $m=1$ (or $4$ and $m=2$).  (The case when $m \geq 2$ for $p >2$ and $m \geq 3$ for $p=2$ will be studied in length in Section~\ref{Section:improve Hodge polygon}.)
Note that  this includes the case when $\psi$ is trivial.
For simplicity, we assume that the condition \eqref{E:buzzard condition} holds for $U$ replaced by $\prod_{l \neq p} U_l \times \GL_2(\ZZ_p)$.  In particular, $(p+1)|t$ if $p>2$ and $6|t$ if $p =2$.
\begin{enumerate}
\item
When $\psi$ is non-trivial of conductor $p$ or $4$,
we know that the $U_p$-slopes on $S^D_{k+1}(U; \psi )$ are exactly given by $k$ minus the $U_p$-slopes on $S^D_{k+1}(U; \psi^{-1} )$, by Atkin--Lehner theory (see Proposition~\ref{P:Up pair to p}  for the proof in the case of $k=1$, and the general case being similar).
Thus, applying Theorem~\ref{T:weak Hodge polygon} to $S^D_{k+1}(U; \psi) $ and $ S^D_{k+1}(U; \psi^{-1})$ and using the fact above, we can improve the lower bound in Theorem~\ref{T:weak Hodge polygon} of the Newton polygon for the $U_p$-action on the \emph{direct sum} $S^D_{k+1}(U; \psi) \oplus S^D_{k+1}(U; \psi^{-1})$, which must lie above the polygon with slopes
\begin{itemize}
\item (if $k$ is even) $0,\, 1, \, \dots \,, \frac k2 -1, \, \frac k2+1, \, \frac k2+2, \, \dots, \, k$, each with multiplicity $2t$;
\item (if $k$ is odd)
$0,\, 1, \, \dots \,, \frac {k-1}2, \, \frac {k+1}2, \, \frac {k+3}2, \, \dots, \, k$, each with multiplicity $2t$, except the slopes $\frac {k-1}2$ and $\frac {k+1}2$ each has multiplicity $t$.
\end{itemize}
We do not know how to improve the bound on the $U_p$-slopes of each individual $S_{k+1}^D(U; \psi)$.
\item
When $\psi$ is the trivial character, $S^D_{k+1}(U;  \textrm{triv})$ is the direct sum of the $p$-old part $S^D_{k+1}(U;  \textrm{triv})^{p\textrm{-old}}$ and the $p$-new part $S^D_{k+1}(U;  \textrm{triv})^{p\textrm{-new}}$.
Given our earlier hypothesis on $U$, we have
\[
\dim S^D_{k+1}(U;  \textrm{triv})^{p\textrm{-old}} = \tfrac{2}{p+1}kt, \quad \textrm{and} \quad \dim S^D_{k+1}(U;  \textrm{triv})^{p\textrm{-new}} = \tfrac{p-1}{p+1}kt.
\]
The eigenvalues of $U_p$-action on $S^D_{k+1}(U; \textrm{triv})^{p\textrm{-new}}$ all have valuation $(k-1)/2$;
whereas the eigenvalues of $U_p$-action on $S^D_{k+1}(U; \textrm{triv})^{p\textrm{-old}}$ can be paired so that the product of each pair has valuation $k$, according to the property of $p$-stabilization.
Thus the total $U_p$-slope on $S_{k+1}^D(U; \mathrm{triv})$ is
\begin{equation}
\label{E:total slope when char is triv}
\tfrac{1}{p+1}kt \cdot k + \tfrac{p-1}{p+1}kt \cdot \tfrac {k-1}2 = \tfrac{pk-p+k+1}{2(p+1)} kt.
\end{equation}
We claim that  the Newton polygon for the $U_p$-action on $S^D_{k+1}(U; \textrm{triv})$ lies above the polygon with slopes
\begin{itemize}
\item
$0, 1, \dots, [\frac{k}{p+1}]-1$, each with multiplicity $t$,
\item
$[\frac{k}{p+1}]$ with multiplicity $\frac{kt}{p+1} - [\frac{k}{p+1}] t$,
\item
$\frac{k-1}2$ with multiplicity $\frac{(p-1)kt}{p+1}$,
\item
$k - [\frac{k}{p+1}]$ with multiplicity $\frac{kt}{p+1} - [\frac{k}{p+1}] t$, and
\item
$k - [\frac{k}{p+1}]+1, k- [\frac{k}{p+1}]+2, \dots, k$, each with  multiplicity $t$.
\end{itemize}
Indeed, the fact that the lower bound over the interval $[0, \frac{kt}{p+1}]$ follows from Theorem~\ref{T:weak Hodge polygon}.
The lower bound over the interval $[\frac{pkt}{p+1}, kt]$ follows from the above lower bound together with the property of $p$-stabilization (as we know the total $U_p$-slopes as computed in \eqref{E:total slope when char is triv}).
The fact that the dimension of the $p$-new forms is $\frac{p-1}{p+1}kt$ and the $p$-stabilization property implies that the $n$th slope with $n \in (\frac{kt}{p+1}, \frac{pkt}{p+1}]$ of $S_{k+1}^D(U;\mathrm{triv})$ is greater than or equal to $\frac{k-1}2$. So the Newton polygon of the $U_p$-action on $S_{k+1}^D(U; \mathrm{triv})$, over the interval $[\frac{kt}{p+1}, \frac{pkt}{p+1}]$ lies above the segment with slope $\frac{k-1}{2}$ starting at the known lower bound at the point $x = \frac{kt}{p+1}$.
This proves the claim.
\end{enumerate}

We point out that the bounds in both cases share the same end point with the actual Newton polygon of the $U_p$-action.
Moreover, the distance of this end point and the vertex $(kt, \frac{k(k-1)}{2}t)$ given by Theorem~\ref{T:weak Hodge polygon} is linear in $k$. So Theorem~\ref{T:weak Hodge polygon} is already a quite sharp bound in this sense.
\end{remark}

\begin{remark}
\label{R:heuristic gouvea}
Keep the setup as in Remark~\ref{R:improve Hodge bound classical forms}(2) and consider the case of trivial character now.
 Gouv\^ea \cite{gouvea} has computed many numerical examples\footnote{Rigorously speaking, Gouv\^ea \cite{gouvea} worked with actual modular forms, but we expect the analogue of his conjecture applies in this case.} to support his expectation of the distribution of $U_p$-slopes on $S^D_{k+1}(U; \textrm{triv})$.
If one uses $a_1(k)\leq \dots \leq a_{kt/(p+1)}(k)$ to denote the lesser slopes on the space of $p$-old forms,
Gouv\^ea conjectured that the distribution given by the numbers
\[
a_1(k) / k,\, a_2(k)/k,\, \dots, \,a_{kt/(p+1)}/k, \textrm{ as } k \to \infty,
\]
converges to a uniform distribution on $[0,\frac{1}{p+1}]$.

In view of the discussion above, this conjecture can be reinterpreted as: the Newton polygon of $U_p$-action on $S^D_{k+1}(U; \mathrm{triv})$ ``stays close" to the lower bound given in Remark~\ref{R:improve Hodge bound classical forms}.
At least, the polygon lower bound provides an inequality for the distribution conjectured by Gouv\^ea.
\end{remark}

\section{An example of explicit computation}
\label{Section:explicit example}
In this section, we give an example of by-hand computation of the $U_p$-slopes for a particular definite quaternion algebra, a prime number $p$, and a level structure.
This case was  considered earlier by Jacobs \cite{jacobs}, a former student of Buzzard, in his thesis.
Unfortunately, Jacobs relied too much on  the computer and hence made the computation unaccessable to people who are interested in checking for patterns.
We reproduce a variant of this computation to serve as a key toy model of our various proofs.
We hope that this hand-on computation can inspire the readers to further develop this technique.

\subsection{The quaternion algebra}
In this section, we consider the quaternion algebra $D$ which ramifies exactly at $2$ and $\infty$.  Explicitly, it is
\[
D = \QQ\langle \i,\j\rangle / (\i\j = -\j\i, \i^2 = \j^2 = -1).
\]
Here we use angled bracket to signify the non-commutativity of the algebra.
It is conventional to put $\k = \i\j$.
The maximal order of $D$ is given by
\[
\calO_D  = \ZZ \big\langle\, \i,\, \j,\, \tfrac12 (1+\i+\j+\k) \big\rangle.
\]
The unit group consists of $24$ elements; they are
\[
\calO_D^\times = \big\{ \pm\! 1, \pm \i, \pm \j, \pm \k, \tfrac12(\pm 1 \pm \i \pm \j \pm \k)\; \big\}.
\]

\subsection{Level structure}
Our distinguished prime $p$ is $3$.
Put $D_f = D\otimes \AAA_f$.
For each $l \neq 2$, we identify $D \otimes \QQ_l$ with $\rmM_2(\QQ_l)$.
For $l = 2$, we use $D^\times(\ZZ_2)$ to denote the maximal compact subgroup of $  (D \otimes \QQ_2)^\times$.
We consider the following open compact subgroup of $D^\times_f$:
\begin{equation}
\label{E:level structure example}
U = D^\times(\ZZ_2) \times \prod_{l \neq 2,3}\GL_2(\ZZ_l) \times \begin{pmatrix}
\ZZ_3^\times & \ZZ_3\\
9 \ZZ_3 & 1+ 3\ZZ_3
\end{pmatrix}.\footnote{Our choice of the level structure is slightly different from \cite{jacobs}, who uses the $\Gamma_1(9)$-level structure.  Here $\Gamma_1(9)$ is defined in the same way as  \eqref{E:level structure example} but with the lower right entry of the last factor replaced by $1+9\ZZ_3$.  As a result, Jacobs had to go through an additional factorization to get the same answer.}
\end{equation}

We point out that for our choice of $p=3$, this corresponds to $m = 2$ in Theorem~\ref{T:theorem B}; so it is not literally covered by it.

\begin{notation}
Let $\nu_3$ denote the square root of $-2$ that is congruent to $1$ modulo $3$.  We have a $3$-adic expansion
\[
\nu_3 = 1+ 3+2 \cdot 3^2 + 2 \cdot 3^5 + 3^7 + \cdots.
\]
We choose the isomorphism between $D \otimes \QQ_3$ and $\rmM_2(\QQ_3)$ so that
\[
1 \leftrightarrow \begin{pmatrix}
1 &0 \\ 0 & 1
\end{pmatrix}, \quad \i \leftrightarrow
\begin{pmatrix}
\nu_3 & 1\\ 1& -\nu_3
\end{pmatrix}, \quad
\j \leftrightarrow \begin{pmatrix}
0 & -1 \\ 1 & 0
\end{pmatrix}, \textrm{ and } \k \leftrightarrow \begin{pmatrix}
1 & -\nu_3 \\ -\nu_3 &- 1
\end{pmatrix}
.
\]
\end{notation}

\begin{lemma}
\label{L:honest decomposition V(9)}
The following natural map is bijective.\footnote{In \cite{jacobs}, $D^\times_f$ is written as the disjoint union of three double cosets, which in fact corresponds to the double coset decomposition of $U$ over $\Gamma_1(9)$.}
\[
\xymatrix@R=0pt{
D^ \times \times  U \ar[r] & D_f^\times\\
(\delta, u) \ar@{|->}[r] & \delta u.
}
\]
\end{lemma}
\begin{proof}
This is of course coincidental for our choices of $D$, $p$ and $U$.  We first observe that $D_f^\times = D ^\times \cdot U_\mathrm{max}$ (see \cite[Lemma~1.22]{jacobs}), where $U_\mathrm{max}$ is a maximal open compact subgroup of $(D\otimes \AAA_f)^\times$, defined using the same equation as in \eqref{E:level structure example} except the factor at $3$ is replaced by $\GL_2(\ZZ_3)$.  Taking into account of the duplication, we have
\[
D_f^\times = D ^\times \times_{\calO_D^\times}  U_\mathrm{max}.
\]
So it suffices to check that the image of $\calO_D^\times$ in $\GL_2(\ZZ_3)$ turns out to form a coset representative of $ U_\mathrm{max} / U$.
This can be checked easily by hand.  (See the proof of \cite[Theorem~2.1]{jacobs} for the list of residues of $\calO_D^\times$ when taking modulo $9$.)
\end{proof}

\begin{corollary}
Let $\psi$ be a continuous character of $\ZZ_3^\times$ of conductor $9$ such that $\psi(-1)=1$ and let $\kappa = x \langle x\rangle ^w\psi$ with $w\in \calO_{\CC_3}$ be a character considered in Example~\ref{Ex: characters on weight space}.
Then evaluation at $1$ induces an isomorphism $S^{D,\dagger}(U; \kappa) \cong \calA$.
\end{corollary}

\begin{lemma}\label{lemma1}
For the case considered in this section, the map $\gothU_3$ in Proposition~\ref{P:explicit Up} is given by
$\gothU_3 = ||_\kappa \delta_1 + ||_\kappa \delta_2 + ||_\kappa \delta_3$, where
\[
\delta_1 =
-1 + \i - \j, \ \delta_2= \tfrac12(1+\i+3\j+\k),\ \textrm{and }\delta_3= \tfrac12(1-3\i-\j-\k).
\]
The images of $\delta_1, \delta_2,\delta_3$ in $\GL_2(\ZZ_3)$ are given by
\[
\begin{pmatrix}
\nu_3-1 & 2 \\ 0 & -1- \nu_3
\end{pmatrix}, \quad
\begin{pmatrix}
1+\frac{\nu_3}2 & -1-\frac{\nu_3}{2} \\ 2-\frac{\nu_3}{2} & -\frac{\nu_3}2
\end{pmatrix}, \quad \textrm{and }
\begin{pmatrix}
-\frac{3\nu_3}{2} & -1+\frac{\nu_3}{2} \\ -2+\frac{\nu_3}{2} & 1+ \frac{3\nu_3}2
\end{pmatrix}.
\footnote{One compares these matrices with the ones appearing after \cite[Lemma~2.5]{jacobs}. Jacobs has a different normalizations which could be removed if one wishes.  Also, we think his matrices involving $v_1^{-1}$ are not correct; this error is however fixed on the next page of {\it loc. cit.}.}
\]
Modulo $9$, they are
\[
\begin{pmatrix}
           3 & 2 \\ 0 & 4
         \end{pmatrix}, \quad\begin{pmatrix}
          3 & 6 \\ 0 & 7
             \end{pmatrix}, \quad \textrm{and }\begin{pmatrix}
          3 & 1 \\ 0 & 7
             \end{pmatrix}.
\]
\end{lemma}
\begin{proof}
We follow the computation in Proposition~\ref{P:explicit Up}.
We need to compute
\[
U_3(\varphi)(1) = \sum_{j=1}^3 \varphi(v_j^{-1})||_\kappa v_j, \quad \textrm{for }v_j = \big(
\begin{smallmatrix}
3&0\\ 9j &1
\end{smallmatrix}
\big)
\]
By Lemma~\ref{L:honest decomposition V(9)}, we can write each $v_j^{-1}$ uniquely as $\delta_j^{-1} u_j$ for $\delta_j \in D^\times$ and $u_j \in U$.
Then
\[
\varphi(v_j^{-1})||_\kappa v_j = \varphi(1)||_\kappa (u_{j,3}v_j) = \varphi(1)||_\kappa \delta_{j,3},
\]
where $u_{j,3}$ and $\delta_{j,3}$ denote the $3$-components of $u_j$ and $\delta_j$, respectively.
On the other hand, we have
\[
\delta_j = u_jv_j \in D^\times \cap U v_j \subseteq D^\times\cap U\big( \begin{smallmatrix}
3 & 0 \\ 0 & 1
\end{smallmatrix}\big) U = D^\times \cap U_1(3) \big( \begin{smallmatrix}
3 & 0 \\ 0 & 1
\end{smallmatrix}\big),
\]
where $U_1(3)$ is defined as $U$ in \eqref{E:level structure example} except the last factor is replaced by $\big(
\begin{smallmatrix}
\ZZ_3^\times & \ZZ_3 \\ 3 \ZZ_3 & 1+ 3\ZZ_3
\end{smallmatrix}
\big)$.
If we put $\delta_j = \delta'_j (1 - \i + \j)$, then we have
\begin{align}
\nonumber
\delta'_j &\in D^\times \cap  U_1(3)\big(\begin{smallmatrix}
3 & 0 \\ 0 & 1
\end{smallmatrix}\big)(1-\i +\j)^{-1}  =  D^\times \cap U_1(3) \big(\begin{smallmatrix}
1+ \nu_3 & 2 \\ 0 & (1-\nu_3)/3
\end{smallmatrix}\big)\\
\nonumber
&=D^\times \cap U_1(3) \big(\begin{smallmatrix}
5 & 2 \\ 0 & 2
\end{smallmatrix}\big) = \big\{
-1, \ \tfrac12(1+\i+\j-\k),\ \tfrac12(1-\i-\j+\k)
\big\}
\end{align}
The last equality follows from looking at the list of $\calO_D^\times$ modulo $3$.  (In the notation of Jacobs' thesis \cite{jacobs}, this set is $\{-1, u_5, -u_8\}$.)

It is then clear that all $\delta_j$'s are among the collections of the above right-multiplied by $1-\i+\j$.  The rest of the lemma is straightforward.
\end{proof}

The main theorem of Jacobs' thesis \cite{jacobs} is the following.

\begin{theorem}[Jacobs]
Let $\psi$ be  a character of $\ZZ_3^\times$ of conductor $9$ such that $\psi(-1)=1$.
We consider the characters $\kappa = x\langle x\rangle ^w\psi$ ($w \in \calO_{\CC_3}$) as in Example~\ref{Ex:classical weight}.
The slopes of the $U_3$-operator acting  on $S^{D,\dagger}(U; \kappa)$ are $\frac{1}{2},1+\frac{1}{2},2+\frac{1}{2},3+\frac{1}{2},\dots$.
\end{theorem}
\begin{proof}
Put $\xi = \psi(4)$; it is a primitive third root of unity.  Then $\psi(7) = \xi^2$.
Put $\pi = \xi-1$ so that $v(\pi) = \frac 12$.
Let $H_{\gothU_3}(x,y)$ denote the generating series of the Hecke operator acting on $S^{D,\dagger}(U; \kappa) \cong \calA$.
By Lemma~\ref{lemma1}, the map $\gothU_3$ is given as $
    \gothU_3 = ||_\kappa \delta_1 + ||_\kappa \delta_2 + ||_\kappa \delta_3$
for the elements $\delta_1, \delta_2, \delta_3$ given therein.
By Lemma~\ref{lemma1}, we have
\begin{align*}
 H_{\gothU_3}(\frac{1}{3\pi} x,\pi y) &\equiv
\frac{\psi(4)4^w}{4-3 \frac{1}{3\pi} x \cdot \pi x -2\pi y}+ \frac{\psi(7)7^w}{7-3  \frac{1}{3\pi}x \cdot \pi y-6 \pi y}
        + \frac{\psi(7)7^w}{7-3  \frac{1}{3\pi} x  \cdot\pi y-\pi y}
        \\
     & \equiv \frac{\xi}{4-xy-2\pi y}+ \frac{\xi^2}{7-xy-6\pi y}+ \frac{\xi^2}{7-xy-\pi y}
     \\
     &\equiv
 \frac{1+\pi}{1-xy+\pi y}+ \frac{(1+\pi)^2}{1-xy}+ \frac{(1+\pi)^2}{1-xy-\pi y} \pmod 3
\end{align*}
It is now straightforward to check that this is congruent to $\dfrac{2\pi}{1-xy}$ modulo $3$.
In other words, the matrix $\Diag(\frac{1}{3\pi}) \cdot \gothU_3 \cdot \Diag(\pi)$ is congruent modulo $3$ to $2 \pi \cdot  I_\infty$, where $I_\infty$ is the infinite identity matrix.
It follows from this easily that the slopes of the $U_3$-operator acting  on $S^{D,\dagger}(U; \kappa)$ are $\frac{1}{2},1+\frac{1}{2},2+\frac{1}{2},3+\frac{1}{2},\dots$.
\end{proof}

\section{Improving the lower bound}
\label{Section:improve Hodge polygon}

The key to obtain a strong result on $U_p$-slopes is to improve the lower bound in Theorem~\ref{T:weak Hodge polygon} so that it agrees with the Newton polygon for sufficiently many points.

\begin{hypothesis}
\label{H:psi conductor pm}
In this and the next section, we retain the notation from Section~\ref{Section:computation of Up} to work with a general definite quaternion algebra $D$ (which splits at $p$).
We fix an integer  $m \geq 4$.
By writing $\psi_m$, we always mean a finite continuous character of $\ZZ_p^\times$ of conductor $p^m$. Let $E$ be a finite extension of $\QQ_p(\zeta_{p^{m-1}})$.
The level structure at $p$ is always taken to be $U_0(p^m)$ with the same number $m$. Some of the results (perhaps after modification) may hold for smaller $m$; see Remark~\ref{R:m=3}.

We assume that $\psi(-1) = 1$.
\end{hypothesis}

\subsection{Facts about classical automorphic forms}
\label{S:classical automorphic forms}
To avoid future confusion, we must clarify how twisting an automorphic representation $\pi$ (of weight $2$) by a central character $\eta: (\ZZ_p/p^m\ZZ_p)^\times \to E^\times$ works, in an explicit way.
We consider the following space of classical automorphic forms
\begin{align}
\label{E:S2 classical twist}
S_2^D(U; \psi_m; \eta): = \big\{ \varphi: D^\times_f \to E \, \big| \, \varphi(\delta  gu&) =  \eta(ad) \psi_m(d) \varphi(g) \textrm{ for any }\delta \in D^\times,\\
\nonumber
& g \in D^\times_f, \textrm{and } u\in U \textrm{ with }u_p = \big(\begin{smallmatrix}
a&b\\c&d
\end{smallmatrix} \big) \big\}.
\end{align}
It carries an action of Hecke operators $T_l$ (for $l\notin \calS$) and $U_p$ just as defined in Subsection~\ref{S:Hecke action}, except multiplied by $\eta(l)$ and $1$, respectively.

Recall that (after making a finite extension of the coefficient field $E$), we have a decomposition of automorphic representations under the actions of all Hecke operators $T_l$ ($l \notin \calS$):
\begin{equation}
\label{E:spectral decomposition over C}
S_2^D(U; \psi_m; \eta) = \bigoplus_\pi V(\pi),
\end{equation}
where the sum is taken over all automorphic representations $\pi$ of $\GL_2(\AAA^\infty)$ of weight $2$.  We say that $\pi$ \emph{appears} in $S_2^D(U; \psi_m ; \eta)$ if the corresponding space $V(\pi) \neq 0$.

We may view $\eta$ as a Hecke character of $\AAA^\times$ via the identification
\[
\QQ^\times \Big\backslash \AAA^\times \Big/ \Big( \RR^\times_{>0}(1+p^m\ZZ_p)^\times\prod_{l \neq p} \ZZ_l^\times\Big) \cong (\ZZ_p/p^m\ZZ_p)^\times.
\]
Write $\pi \otimes (\eta \circ \det)$ for the tensor product of the automorphic representations of $\GL_2(\AAA)$; it has central character $\omega_\pi \eta^2$, where $\omega_\pi$ is the central character of $\pi$.
For each $l \notin \calS$, the $T_l$-eigenvalue on the spherical vector at $l$ for $\pi \otimes (\eta \circ \det)$ is $\eta(l)$ times that for $\pi$.

It is clear from this construction that $\pi$ appears in $S_2^D(U; \psi_m)$ if and only if $\pi \otimes (\eta \circ \det)$ appears in $S_2^D(U; \psi_m; \eta)$.
In fact we have a canonical isomorphism of modules of Hecke operators $T_l$ for $l \notin \calS$ and $U_p$
\begin{equation}
\label{E:untwist}
S_2^D(U; \psi_m) \otimes (\eta \circ \det) \cong S_2^D(U; \psi_m; \eta),
\end{equation}
where $T_l$ for $l \notin \calS$ acts on the factor $(\eta \circ \det)$ by multiplication by $\eta(l)$ and $U_p$ acts trivially.
We must point out that
$S_2^D(U; \psi_m; \eta)$ is \emph{genuinely different} from  $S_2^D(U; \psi_m \eta^2)$ (not even up to twists).

\begin{lemma}
\label{L:only new forms}
Assume Hypothesis~\ref{H:psi conductor pm}.
Then each Hecke eigenform in $S_2^D(U; \psi_m; \omega^r)$ is $p$-new, and the action of $U_p$ on each of $V(\pi)$ in \eqref{E:spectral decomposition over C} is just the scalar multiplication by some $a_p(\pi) \in E$.
Moreover, $v(a_p(\pi)) \in [0,1]$.
\end{lemma}
\begin{proof}
The isomorphism~(\ref{E:untwist}) allows us to assume $r=0$ in \eqref{E:spectral decomposition over C}.
The condition on $\psi_m$ and the level structure ensures that the $p$-component $\pi_p$ of $\pi$ is forced to be a principal series, and has only one-dimensional fixed vector under the action of the group $U_1(p^m) = \big(\begin{smallmatrix} \ZZ_p^\times& \ZZ_p\\p^{m}\ZZ_p&1+p^m\ZZ_p \end{smallmatrix}\big)$.
So $U_p$ acts on $V(\pi)$ in the same way as $U_p$ acts on this one-dimensional fixed vector, by multiplication of some $a_p(\pi)\in E$.
 The norm bound on $v(a_p(\pi))$ follows from the admissibility at $p$ of the Galois representation attached to $\pi$.
\end{proof}

\subsection{Atkin--Lehner involution}
\label{S:Atkin-Lehner}
Recall that weight $2$ automorphic forms are simply functions on $D_f^\times$.
Similar to the case of classical modular forms, we have the following Atkin--Lehner involution map
\begin{equation}
\label{E:Atkin-Lehner}
\xymatrix@R=0pt{
\AL_{\psi_m}: S_2^D(U; \psi_m) \ar[r] & S_2^D(U; \psi_m^{-1}; \psi_m)
\\
\varphi \ar@{|->}[r] & \varphi\big(\bullet \Matrix 01{p^m}0\big).
}
\end{equation}
One checks that the condition
\[
\varphi(gu) = \psi_m(d) \varphi(g) \quad \textrm{ for } u\in U \textrm{ with } u_p = \Matrix abcd
\]
is equivalent to

\[
\varphi\big(gu \Matrix 01{p^m}0 \big) = \psi_m(a) \varphi\big(g\Matrix 01{p^m}0 \big) \quad \textrm{ for } u\in U \textrm{ with } u_p = \Matrix abcd.
\]
This justifies the source and the target of $\AL_{\psi_m}$. It is clear that $\AL_{\psi_m}$ is an isomorphism, preserving the obvious $\calO$-lattice, given by evaluation at the fixed coset representatives $\gamma_i$'s.

\begin{proposition}
\label{P:Up pair to p}
Keep the notation as above.  For $\varphi \in S_2^D(U; \psi_m)$, we have
\begin{equation}
\label{E:Up pair to p}
U_p \circ \AL_{\psi_m} \circ U_p (\varphi) =p \cdot  \AL_{\psi_m} \circ S_p(\varphi),
\end{equation}
where $S_p$ is the automorphism  $\varphi \mapsto \varphi( \bullet \big(\begin{smallmatrix}p^{-1}&0\\0&p^{-1}\end{smallmatrix}\big)\,)$ of $S_2^D(U, \psi_m)$ given by shifting the variable by an idele at $p$.
\end{proposition}
\begin{proof}

Recall the coset decomposition \eqref{E:Up operator cosets} with coset representatives $v_i = \Matrix p0{ip^m}1$ for $i=0,\dots, p-1$ to define the $U_p$-operator. We have $v_i^{-1} = \Matrix{1/p} 0{-ip^{m-1}} 1$.
Note that the action of $||_{v_i}$ on $E$ is trivial.
We exhibit a direct computation:
\begin{align}\nonumber
U_p \circ \AL_{\psi_m} \circ U_p &(\varphi)(g) = \sum_{i=0}^{p-1} \sum_{j=0}^{p-1} \varphi\big( g v_i^{-1} \Matrix 01{p^m}0 v_j^{-1} \big)\\ \nonumber
&= \sum_{i,j=0}^{p-1} \varphi \Big( g \Matrix{1/p} 0{-ip^{m-1}} 1 \Matrix 01{p^m}0  \Matrix{1/p} 0{-jp^{m-1}} 1 \Big)\\
\label{E:AL calculation}
&= \sum_{i,j=0}^{p-1} \varphi \Big( g\Matrix 01{p^m}0   \Matrix 1{-i/p}0{1/p}  \Matrix{1/p} 0{-jp^{m-1}} 1  \Big).
\end{align}

One can verify the following equality:
\[
\MATRIX 1{-i/p}0{1/p}  \MATRIX{1/p} 0{-jp^{m-1}} 1  = \MATRIX {1/p} 0 {-jp^{m-2}}{1/p} \MATRIX{1+ijp^{m-1}}{-i}{ijp^{2m-2}}{1-ijp^{m-1}}.
\]
Thus, we have
\[
\varphi \Big( g\Matrix 01{p^m}0   \Matrix 1{-i/p}0{1/p}  \Matrix{1/p} 0{-jp^{m-1}} 1  \Big) = \psi_m(1-ijp^{m-1}) \cdot \varphi \Big( g \Matrix 01{p^m}0 \Matrix {1/p} 0 {-jp^{m-2}}{1/p} \Big).
\]
Since $\psi_m$ has conductor exactly $p^m$, as we sum up \eqref{E:AL calculation} over $i$ (for a fixed $j$) all terms cancel to zero unless when $j=0$, the corresponding terms exactly give $p$ copies of
\[
\varphi\big(g \Matrix 01{p^m}0 \Matrix{1/p}00{1/p} \big).
\]
From this, we deduce immediately that
\[
U_p\circ \AL_{\psi_m} \circ U_p(\varphi) = p \cdot \AL_{\psi_m} \circ S_p(\varphi).\qedhere
\]
\end{proof}

\begin{remark}
A similar Atkin--Lehner map exists  for the space of automorphic forms of higher weights $k+1$, except that it does not preserve the natural $\calO$-lattices.  We do not need this generalization in this paper.
\end{remark}

\begin{notation}
We identify the space of weight two classical automorphic forms $S_2^D(U; \psi_m)$ with $\oplus_{i=0}^{t-1} E$ by evaluating at $\gamma_0, \dots, \gamma_{t-1}$.
We use $\gothU_p^\mathrm{cl}(\psi_m)$ and $\gothT_l^\mathrm{cl}(\psi_m)$ to denote the matrices for the Hecke actions of $U_p$ and $T_l$
(for $l \notin \calS$) under the standard basis.

Let $\alpha_0(\psi_m)\leq \dots \leq \alpha_{t-1}(\psi_m)$ denote the slopes of the Hodge polygon of $\gothU_p^\mathrm{cl}(\psi_m)$, in non-decreasing order.  For simplicity, we assume that $E$ contains all powers $p^{\alpha_i(\psi_m)}$.
\end{notation}

\begin{corollary}
\label{C:Hodge polygon less eq 1}
The numbers $\alpha_i(\psi_m)$ belong to $ [0,1]$.
There exists a basis $e_0(\psi_m), \dots, e_{t-1}(\psi_m)$ of $S_2^D(U; \psi_m, \calO) \cong \oplus_{i=0}^{t-1} \calO$ such that the matrix of $U_p$-action is given by a matrix $\gothU_p^\mathrm{cl, \bfe}(\psi_m)$ whose $i$th row is divisible by $p^{\alpha_i(\psi_m)}$.
\end{corollary}
\begin{proof}

The existence of the basis $e_1(\psi_m), \dots, e_t(\psi_m)$ follows from Subsection~\ref{S:Hodge v.s. Newton}(5). We shall prove the first statement now.

It is clear that $\gothU_p^\mathrm{cl}(\psi_m)$ has entries in the integral ring $\calO$.
By Proposition~\ref{P:Up pair to p},
we have
\[
\big(\AL_{\psi_m}^{-1}\circ U_p \circ \AL_{\psi_m}\big) \circ U_p = p\cdot S_p.
\]
Writing this in terms of matrices, we have
\begin{equation}
\label{E:UpUp=pA}
M \cdot
\gothU_p^\mathrm{cl}(\psi_m) = pA,
\end{equation}
where $A \in \GL_t(\calO)$ is the matrix for the action of the central character $S_p$, and $M \in \rmM_t(\calO)$ corresponds to the action of $U_p$ on $S_2^D(U; \psi_m^{-1}; \psi_m)$. (Note that $\AL_{\psi_m}$ is an isomorphism preserving the natural $\calO$-lattices.)
By Subsection~\ref{S:Hodge v.s. Newton}(4), we can write $\gothU_p^\cl(\psi_m) = BDC$ with $B, C \in \GL_t(\calO)$ and $D$ diagonal, so that the valuations of the diagonal entries of $D$ are exactly $\alpha_0(\psi_m), \dots, \alpha_{t-1}(\psi_m)$.
So \eqref{E:UpUp=pA} can be rewritten as
\[
M = AC^{-1}(pD^{-1}) B^{-1},
\]
where $AC^{-1}, B \in \GL_t(\calO)$.
By Subsection~\ref{S:Hodge v.s. Newton}(4), this means that the slopes of the Hodge polygon of $M$ are given by $1-\alpha_i(\psi_m)$. Since $M$ and $\gothU_p^\cl(\psi_m)$ both have integral entries, both $\alpha_i(\psi_m)$ and $1-\alpha_i(\psi_m)$ are non-negative. Thus $\alpha_i(\psi_m) \in [0,1]$.
\end{proof}

\subsection{A variant of the Atkin--Lehner map}
\label{S:variant of duality pairing}
For a purely technical reason, we need a generalization of the Atkin--Lehner map \eqref{E:Atkin-Lehner} of automorphic forms whose weights are infinitesimal deformations of weight $2$.

Consider the ring $\calO[w]/(p^2w)$ consisting of ``polynomials" $f(w) = a_0+ a_1w + \cdots $ where the coefficients $a_i$ with $i \geq 1$ are only meaningful modulo $p^2 \calO$.

Let $w$ be an indeterminant.
Note that we have a character
\[
\xymatrix@R=0pt{
\psi_{m,w}: \ \big({\begin{smallmatrix} \ZZ_p^\times& \ZZ_p\\ p^{m}\ZZ_p&\ZZ_p^\times \end{smallmatrix}}\big)
\ar[rr] && \big( \calO[w]/(p^2w) \big)^\times
\\
\big({\begin{smallmatrix} a&b\\ c&d \end{smallmatrix}}\big)
\ar@{|->}[rr] && \psi_m(d) \langle d\rangle ^w,
}
\]
where $\langle d \rangle: = d \omega^{-1}(d)$ is defined as before and $\langle d\rangle^w = 1+ (\langle d \rangle -1)w \in \calO[w]/(p^2w)$. Note that it is important here to consider torsion coefficients, otherwise, $\psi_{m,w}$ is not a homomorphism of groups.
We point out that the image of $\psi_{m,w}$ in fact lands in the \emph{subring}
\begin{equation}
\label{E:torsion image}
\calO + p \calO /p^2\calO \cdot w
\end{equation}
of $\calO[w]/(p^2w)$.
We think of $\psi_{m,w}$ as certain deformation of the character $\psi_m$.

We introduce the following deformed version of classical automorphic forms:
\begin{equation}
\label{E:S2 deformation}
S_2^D(U; \psi_{m,w}): = \Big\{ \varphi: D^\times_f \to \calO[w]/(p^2w) \, \Big| \,
\begin{array}{c}
 \varphi(\delta gu) =  \psi_m(d)\langle d\rangle^w \varphi(g) \textrm{ for any }
\\
 \delta \in D^\times, g \in D^\times_f,
\textrm{and } u\in U \textrm{ with }u_p = \big(\begin{smallmatrix}
a&b\\c&d
\end{smallmatrix} \big)
\end{array}
 \Big\}.
\end{equation}
This is a finite free $\calO[w]/(p^2w)$-module which carries linear actions of $T_l$ (for $l \notin \calS$) and $U_p$ in the natural way.

Abstractly, we can identify $S_2^D(U; \psi_{m, w})$ with $S_2^D(U; \psi_m; \calO) \otimes_\calO \calO[w]/(p^2w)$ by identifying the evaluations at $\gamma_i$'s.
Then the elements $e_0(\psi_m), \dots, e_{t-1}(\psi_m)$ in Corollary~\ref{C:Hodge polygon less eq 1} gives rise to a basis of $S_2^D(U; \psi_{m, w})$ over $\calO[w]/(p^2w)$.
Let $\gothU_p^{\cl, \bfe}(\psi_{m, w}) \in \rmM_t \big(\calO[w]/(p^2w) \big)$ denote the matrix for the $U_p$-action on $S_2^D(U; \psi_{m,w})$ with respect to this basis.
Since $\psi_{m,w}$ takes value in the \emph{subring} \eqref{E:torsion image}, all entries of $\gothU_p^{\cl, \bfe}(\psi_{m, w})$ belongs to this subring $\calO+p\calO/p^2\calO \cdot w$.
It follows that the $i$th row  of $\gothU_p^{\cl, \bfe}(\psi_{m, w})$ belongs to $p^{\alpha_i(\psi_m)}\calO[w] / p^2w \calO[w] \subseteq \calO[w]/(p^2w)$.  This is true because all $\alpha_i(\psi_m) \in [0,1]$; so the coefficient on the variable $w$ belongs to $p\calO / p^2\calO \subseteq p^{\alpha_i(\psi_m)}\calO / p^2\calO$.

The following technical proposition will be important for us later.

\begin{proposition}
\label{P:determine of Up deformed}
Let $\overline \gothU_p^{\cl, \bfe}(\psi_{m,w}) \in \rmM_t(\FF[w])$ denote the reduction modulo $\varpi$ of the matrix given by dividing the $i$th row of $\gothU_p^{\cl, \bfe}(\psi_{m,w})$ by $p^{\alpha_i(\psi_m)}$.\footnote{While the division by $p^{\alpha_i(\psi_m)}$ in the ring $\calO/p^2\calO[w]$ is not well-defined, the reduction modulo $\varpi$ of the quotient is a well-defined element in $\FF[w]$.}
Then
\[
\det \big( \overline \gothU_p^{\cl, \bfe}(\psi_{m,w}) \big) \in \FF^\times,
\]
i.e., it is invertible as a matrix in $\rmM_t(\FF[w])$.
\end{proposition}
\begin{proof}
We write $S_2^D(U; \psi_{m,w}^{-1}; \psi_{m,w})$ for the space defined similar to \eqref{E:S2 deformation} except that $\psi_m(d) \langle d \rangle ^w$ is replaced by $\psi_m(a) \langle a \rangle ^w$.
Similar to the case of classical automorphic forms, we define an Atkin--Lehner map
\[
\xymatrix@R=0pt{
\AL_{\psi_{m,w}}: S_2^D(U; \psi_{m,w}) \ar[r] & S_2^D(U; \psi^{-1}_{m,w}; \psi_{m,w})
\\
\varphi \ar@{|->}[r] & \varphi\big(\bullet \Matrix 01{p^m}0\big).
}
\]
One can similarly check that $\AL_{\psi_{m,w}}$ is a well-defined isomorphism.

Since the proof of Proposition~\ref{P:Up pair to p} is mostly tautological and $\psi_{m,w}(1+ap^{m-1}) = \psi_m(1+ap^{m-1})$ for $a \in \ZZ$,
its proof may be transported verbatim to our setup to show
\begin{equation}
\label{E:Up pair to p deformed}
(\AL_{\psi_{m,w}}^{-1} \circ U_p \circ \AL_{\psi_{m,w}}) \circ U_p (\varphi) =p \cdot  S_p(\varphi).
\end{equation}

Let $B \in \GL_t(\calO)$ denote the change of basis matrix from the basis given by evaluation at $\gamma_i$'s to the basis $e_0(\psi_m), \dots, e_{t-1}(\psi_m)$.
Then \eqref{E:Up pair to p deformed} gives
\[
M_w B^{-1} \gothU_p^{\cl, \bfe}(\psi_{m,w})B = p A_w,
\]
where $A_w \in \GL_t(\calO[w]/(p^2w)) $ is the matrix for the action of $S_p$ on $S_2^D(U; \psi_{m,w})$, and $M_w \in \rmM_t(\calO[w]/(p^2w))$ is the matrix for the action of $\AL_{\psi_{m,w}}^{-1} \circ U_p \circ \AL_{\psi_{m,w}}$.
It follows that
\begin{equation}
\label{E:Up pair to p deformed matrix version}
BA_w^{-1}M_w B^{-1} \gothU_p^{\cl, \bfe}(\psi_{m,w}) =pI \qquad \textrm{as matrices in }\rmM_t(\calO[w]/(p^2w)).
\end{equation}

We claim that all entries on the $i$th row of the matrix $N = BA_w^{-1}M_w B^{-1}$ belongs to $p^{1-\alpha_i(\psi_m)}\calO[w] / p^2w\calO[w] \subseteq \calO[w]/(p^2w)$.
For a matrix $L \in \rmM_t(\calO[w]/(p^2w))$, we write $L|_{w=0}$ for the matrix whose entries are the image of the corresponding entires of $L$ under the map $\calO[w]/(p^2w) \to \calO[w]/(w) \cong \calO$.

We first show that the $i$th row of  $N|_{w=0}$
belongs to $p^{1-\alpha_i(\psi_m)}\calO$.
For this, we write $\gothU_p^{\cl, \bfe}(\psi_{m,w})|_{w=0} = \gothU_p^{\cl, \bfe}(\psi_m)= \Diag\{p^{\alpha_0(\psi_m)}, \dots, p^{\alpha_{t-1}(\psi_m)}\}\cdot C$ for some matrix $C \in \rmM_t(\calO)$.
Since $\alpha_0(\psi_m), \dots, \alpha_{t-1}(\psi_m)$ are the slopes of Hodge polygon of the matrix $\gothU_p^{\cl, \bfe}(\psi_{m})$, the determinant of the matrix $C$ belongs to $\calO^\times$. Thus, $C \in \GL_t(\calO)$.
Using this and \eqref{E:Up pair to p deformed matrix version}, we deduce
\[
N|_{w=0} \Diag\{p^{\alpha_0(\psi_m)}, \dots, p^{\alpha_{t-1}(\psi_m)}\}\cdot C = pI,
\]
and hence
\[
N|_{w=0}  = C^{-1} \cdot \Diag\{p^{1-\alpha_0(\psi_m)}, \dots, p^{1-\alpha_{t-1}(\psi_m)}\}.
\]
So the $i$th column of $N|_{w=0}$ belongs to $p^{1-\alpha_i(\psi_m)} \calO$.

Now, we observe that the matrix $N$ is a product of matrices with entries in the subring $\calO + p\calO/p^2\calO  \cdot w$, so the entires of the $i$th column of $N$ belongs to $p^{1-\alpha_i(\psi_m)}\calO[w] / p^2w\calO[w]$, as the coefficients on $w$ automatically belongs to $p\calO \subseteq p^{1-\alpha_i(\psi_m)}\calO$.

We use $\overline N \in \rmM_t(\FF[w])$ to denote the reduction modulo $\varpi$ of the matrix given by dividing the $i$th column of $M$ by $p^{1-\alpha_i(\psi_m)}$.
It then follows that
\[
\overline N \cdot
\overline \gothU_p^{\cl, \bfe}(\psi_{m,w}) = I \qquad \textrm{ as matrices in }\rmM_t(\FF[w]).
\]
So $ \overline \gothU_p^{\cl, \bfe}(\psi_{m,w}) \in \GL_t(\FF[w])$ and its determinant belongs to $\FF^\times$.
\end{proof}

\begin{notation}
\label{N:character kappa}
Let $\psi_m$ be as in Hypothesis~\ref{H:psi conductor pm}.
We use $\calW(x\psi_m; p^{-1})$ to denote the closed disk of radius $p^{-1}$ ($\frac 14$ in case $p=2$)\footnote{We apologize for the confusing notation when $p=2$.} centered at $x\psi_m$ in the weight space.
This disk corresponds to all characters of the form $x\psi_m \langle \cdot \rangle^w$ for $w \in\calO_{\CC_p}$, in particular, including classical characters $x^k\psi_m\omega^{1-k}$ for $k \geq 1$.

We take $A^\circ$ to be the Tate algebra $\calO\langle w\rangle$ and $A $ to be $ E\langle w\rangle$.
We identify $\Max(A) = \Max(E\langle w\rangle)$ with the disk $\calW(x \psi_m; p^{-1})$ so that the universal character
$\kappa: \Gamma \to E\langle w\rangle^\times$ is given by
\[
\kappa(a): = a \psi_m(a)\langle a\rangle^{w}.
\]
Here the expression $\langle a\rangle^w$  is understood as $(1+2pb)^w = \sum_{n\geq 0} (2pb)^n \binom wn \in 1+2pw\ZZ_p\langle w\rangle$, if $\langle a\rangle  = 1+2pb  $.
\end{notation}

\subsection{A variant of the space of overconvergent automorphic forms}
For a technical reason, it is more convenient to consider a variant of the space of overconvergent automorphic forms, with coefficients in $\calB : = A\langle p z\rangle = E\langle w, pz \rangle \subset A\widehat \otimes \calA$. 

Recall that the right action $||_\kappa \gamma$ of $\gamma = \big(
\begin{smallmatrix}
a&b\\c&d
\end{smallmatrix} \big) \in \Sigma_0(p^m)$ on $A \widehat \otimes\calA$ is given by
\begin{equation}
\label{E:right action by gamma}
(h||_{\kappa}\gamma)(z): = \frac{\kappa(cz+d)}{cz+d} h\big( \frac{az+b}{cz+d}\big) = \psi_m(d)\langle d\rangle^w \big(1+\frac cdz\big)^{w}h\big( \frac{az+b}{cz+d}\big)
 \textrm{ for }h(z) \in \calA.
\end{equation}
Since $p \nmid d$ and $ p^m |c$, the expansion of the exponential $(1+\frac cdz)^w$ lands in $\calO\langle w, p^{m-1} z\rangle \subset \calB$.\footnote{This follows from the standard estimate $(1+x)^w = 1+\sum_{n\geq 1}  \binom wn x^n \in \calO\langle w, p^{-1}x\rangle$ (note that  the binomial coefficients are not integral for a free variable $w$.) We will use this estimate freely later in the paper.}
So \eqref{E:right action by gamma} can be applied to an element $h(z) \in \calB$ and gives rise to a right action of $\Sigma_0(p^m)$ on $\calB$.
Therefore, we can define the space of overconvergent automorphic forms with coefficients in $\calB$ (instead of $A \widehat \otimes \calA$):
\[
S^{D,\dagger}_\calB(U; \kappa): = \Big
\{
\varphi: D^\times_f \to\calB\;\Big|\;
\varphi(\delta gu) = \varphi(g)||_{\kappa}u_p, \textrm{ for any }\delta \in D^\times, g \in D_f^\times, u \in U
\Big\};
\]
it is a subspace of $S^{D,\dagger}(U; \kappa)$ (with coefficients in $A \widehat \otimes \calA$).
In explicit forms, we have an isomorphism of Banach spaces $S^{D,\dagger}_\calB(U; \kappa) \xrightarrow{\cong} \oplus_{i=0}^{t-1} \calB$ given by $\varphi\mapsto (\varphi(\gamma_i))_{i=0, \dots, t-1}$.

\begin{notation}
\label{N:matrix for Tl and Up}
We use $\gothU_p^A(\kappa)$ and $\gothT_l^A(\kappa)$ (for $l \notin \calS$) to denote the infinite matrices in Proposition~\ref{P:explicit Up} for the operators $\gothU_p$ and $\gothT_l$ acting on $\oplus_{i=0}^{t-1} A\widehat \otimes \calA$, with respect to the orthonormal basis $1_0, \dots, 1_{t-1}, z_0, \dots, z_{t-1}, z^2_0, \dots$.  Here the subscripts indicate which copy of $\calA$ the element comes from.
We use $\gothU_p^\calB(\kappa)$ and $\gothT_l^\calB(\kappa)$ (for $l \notin \calS$) to denote the infinite matrix for  the operators $\gothU_p$ and $\gothT_l$ acting on $S^{D,\dagger}_\calB(U; \kappa) = \oplus_{i=0}^{t-1} \calB$, with respect to the orthonormal basis $1_0, \dots, 1_{t-1}, pz_0, \dots, pz_{t-1}, p^2z^2_0, \dots$
It is clear from the definition that
\[
\Diag(p^{-1}; t) \gothU_p^A(\kappa) \Diag(p; t) = \gothU_p^\calB(\kappa) , \quad \textrm{and} \quad \Diag(p^{-1}; t) \gothT_l^A(\kappa)  \Diag(p; t) = \gothT_l^\calB(\kappa) .
\]
In particular, $\Char \big(\gothU_p^A(\kappa) , S^{D,\dagger}(U; \kappa)\big) = \Char\big(\gothU_p^\calB(\kappa) , S^{D,\dagger}_\calB(U; \kappa) \big).$  So to understand the $U_p$-slopes on $S^{D, \dagger}(U;\kappa)$, it suffices to look at the $U_p$-slopes on $S^{D,\dagger}_\calB(U;\kappa)$.

\end{notation}

The following lemma gives a key congruence relation between the action of a matrix in $\Sigma_0(p^m)$ on the space of overconvergent automorphic forms and on the space of \emph{classical automorphic forms}.
\begin{lemma}
\label{L:congruence}
Let $\big(\begin{smallmatrix}a&b\\c&d\end{smallmatrix} \big)$ be a matrix in $\Sigma_0(p^m)$ with $v(a) = 0$ or $1$.
Then the matrix for $||_\kappa\big(\begin{smallmatrix}a&b\\c&d\end{smallmatrix} \big)$ acting on $\calB$ (with respect to the basis  $1, pz, p^2z^2, \dots$) belongs to
\[
\begin{pmatrix}
\psi_m(d)\langle d\rangle^w & pA^\circ & p^2 A^\circ & p^3A^\circ &\cdots
\\
p^3A^\circ &  \frac ad\psi_m(d)\langle d\rangle^w + p^2aA^\circ & paA^\circ &p^2aA^\circ &  \cdots
\\
p^4 A^\circ &p^3aA^\circ & (\frac ad)^2\psi_m(d)\langle d\rangle^w + p^2a^2A^\circ& pa^2A^\circ  & \cdots
\\
p^5 A^\circ &p^4 aA^\circ &p^3a^2A^\circ & (\frac ad)^3 \psi_m(d)\langle d\rangle^w+ p^2a^3A^\circ&  \cdots
\\
\vdots & \vdots & \vdots & \vdots & \ddots
\end{pmatrix}
\]
where the $(i,j)$-entry of the matrix is
\begin{itemize}
\item
$(\frac ad)^i \psi_m(d)\langle d\rangle^w+ p^2a^iA^\circ$ if $i=j>0$,
\item
$p^{i-j+2}a^jA^\circ$ if $i>j$, and
\item
$p^{j-i} a^i A^\circ$ if $i<j$.
\end{itemize}
\end{lemma}
\begin{proof}
Note that $(1+  p^{m-1} z)^w \in 1+ p^3zw\calO\langle p z, w\rangle$ since $m \geq 4$.\footnote{It is important that the $zw$ coefficient has valuation \emph{strictly} bigger than $2$.  The case $m=3$ fails exactly at this point. See Remark~\ref{R:m=3}.}  So Proposition~\ref{P:generating series} implies that (note that $ p^m|c,p\nmid d$)
\[
H_{||_\kappa(\begin{smallmatrix}a&b\\c&d\end{smallmatrix} )}(p^{-1}x,py)  = \frac{d\psi_m(d)\langle d \rangle^w(1+p^{-1}\frac cd x)^w}{p^{-1} c x + d - a xy - p by} \in \calO \langle w, py, axy, p^{i+2} x^i; i \in \NN\rangle.
\]
Translating this congruence into the language of matrix and noting that the dominant coefficients on terms $x^iy^i$ come from the expansion of $\frac{d\psi_m(d) \langle d \rangle^w}{d - a xy}$ proves the Lemma.
\end{proof}

Lemma~\ref{L:congruence} implies that the actions of $U_p$ and $T_l$ for $l \notin \calS$ on $S^{D, \dagger}_{\calB}(U; \kappa)$ are ``very close" to their actions on the completed direct sum
\[
\widehat \bigoplus_{n \geq 0} S_2^D(U; \psi_m\omega^{-2n}; \omega^n).
\]
More precisely, we have the following.

\begin{proposition}
\label{P:Tl Up overconvergent equiv classical}
(1)
For $l \notin \calS$, we consider the infinite block diagonal matrix
\[
\gothT_l^{\cl, \infty} : = \Diag\big\{\gothT_l^\cl(\psi_m), \ l \cdot \gothT_l^\cl(\psi_m\omega^{-2}), \ l^2 \cdot \gothT_l^\cl(\psi_m\omega^{-4}), \ \dots \big\}.
\]
Then the difference $\gothT_l^\calB(\kappa) - \gothT_l^{\cl, \infty}$ lies in the error space
\begin{equation}
\label{E:error matrix}
\mathbf{Err}: =
\begin{pmatrix}
p\rmM_t(A^\circ) & p\rmM_t(A^\circ) & p^2 \rmM_t(A^\circ) & p^3\rmM_t(A^\circ) &\cdots
\\
p^3\rmM_t(A^\circ) &  p\rmM_t(A^\circ) & p\rmM_t(A^\circ) &p^2 \rmM_t(A^\circ) &  \cdots
\\
p^4  \rmM_t(A^\circ) &p^3\rmM_t(A^\circ) & p \rmM_t(A^\circ)& p\rmM_t(A^\circ) & \cdots\\
p^5\rmM_t(A^\circ) &p^4\rmM_t(A^\circ) &p^3\rmM_t(A^\circ) &  p\rmM_t(A^\circ)&  \cdots
\\
\vdots & \vdots & \vdots & \vdots & \ddots
\end{pmatrix}
\end{equation}
where the $(i,j)$-block entry of the matrix  is
\begin{itemize}
\item
$ p\rmM_t(A^\circ)$ if $i=j$,
\item
$p^{i-j+2}\rmM_t(A^\circ)$ if $i > j$, and
\item
$p^{j-i}\rmM_t(A^\circ)$ if $i < j$.
\end{itemize}

(2) Similarly, we consider the infinite block diagonal matrix
\[
\gothU_p^{\cl, \infty}: = \Diag\big( \gothU_p^\cl(\psi_m),\ p \cdot \gothU_p^\cl(\psi_m \omega^{-2}),\ p^2\cdot \gothU_p^\cl(\psi_m \omega^{-4}),\ \dots \big).
\]
Then difference $\gothU_p^\calB(\kappa) - \gothU_p^{\cl, \infty}$ lies in the $p$-error space
\begin{equation}
\label{E:Errp}
 \Err_p: =
\begin{pmatrix}
p\rmM_t(A^\circ) & p\rmM_t(A^\circ) & p^2 \rmM_t(A^\circ) & p^3\rmM_t(A^\circ) &\cdots
\\
p^3\rmM_t(A^\circ) &  p^2\rmM_t(A^\circ) & p^2\rmM_t(A^\circ) &p^3 \rmM_t(A^\circ) &  \cdots
\\
p^4  \rmM_t(A^\circ) &p^4\rmM_t(A^\circ) & p^3 \rmM_t(A^\circ)& p^3\rmM_t(A^\circ) & \cdots
\\
p^5\rmM_t(A^\circ) &p^5\rmM_t(A^\circ) &p^5\rmM_t(A^\circ) &  p^4\rmM_t(A^\circ)&  \cdots
\\
\vdots & \vdots & \vdots & \vdots & \ddots
\end{pmatrix},
\end{equation}
where the $(i,j)$-block entry is
\begin{itemize}
\item
$p^{i+1}\rmM_t(A^\circ)$ if $i =j$,
\item
$p^{i+2}\rmM_t(A^\circ)$ if $i>j$, and
\item
$p^j\rmM_t(A^\circ)$ if $i<j$.
\end{itemize}
Moreover, the $(i,i)$-block entry of $\gothU_p^\calB(\kappa)$ is congruent to the matrix $p^i\cdot  \gothU_p^\cl(\psi_{m,w-2i} \omega^{-2i})$ modulo $p^{i+2} \rmM_t(A^\circ)$.
\end{proposition}

\begin{proof}
Note that the global elements $\delta$ appearing in the matrix of $\gothU_p$ or $\gothT_l$ for $l \notin \calS$ in Proposition~\ref{P:explicit Up} are the same for classical or overconvergent automorphic forms for all characters.
So to prove (1) and (2), it suffices to estimate the difference between the actions of each relevant $ \delta_p$ on $S_\calB^{D, \dagger}(U; \kappa)$ and on the completed direct sum $\widehat \bigoplus_{n \geq 0} S_2^D(U; \psi_m\omega^{-2n}; \omega^n)$.  (Note that $l^{r} \cdot \gothT_l^\cl(\psi_m \omega^{-2r})$ is congruent modulo $p$ to the action of $T_l$ on the space of classical automorphic forms $S_2^D(U; \psi_m\omega^{-2r}; \omega^r)$.)

For $l \notin\calS$, Proposition~\ref{P:explicit Up} implies that, for every $\delta_p =\big(\begin{smallmatrix}a&b\\c&d \end{smallmatrix}\big)$ appearing in the expression of  $\gothT_l^\calB(\kappa)$, we have  $a, d \in \ZZ_p^\times$, $b, c \in \ZZ_p$, and  $ad-bc=l$; so we have $a d \equiv l \pmod {p^m}$.
By Lemma~\ref{L:congruence}, $||_\kappa \delta_p$ is, modulo the expression \eqref{E:error matrix} but with $t=1$, congruent to the infinite diagonal matrix with diagonal elements
\[
\psi_m(d)\langle d\rangle^w,\ \psi_m(d)\tfrac ad\langle d\rangle^w,\
\psi_m(d)(\tfrac ad)^2\langle d\rangle^w,\ \dots
\]
 which is the same as
 \[
 \psi_m(d),\ l\tfrac{\psi_m(d)}{d^2},\ l^2
\tfrac{\psi_m(d)}{d^4},\ \dots
\]
 modulo $p$; it is further the same as
 \[\psi_m(d),\ l \psi_m(d) \omega^{-2}(d),\ l^2 \psi_m(d) \omega^{-4}(d),\ \dots
 \]
 modulo $p$.
This is the same as the contribution of $\delta_p$ to the matrix
\[
\gothT_l^{\cl, \infty} = \bigoplus_{r \geq 0} l^r \cdot \gothT_l^\cl(\psi_m\omega^{-2r}).
\]
This concludes the proof of (1).

(2) can be checked similarly:
for each $\delta_p =\big(\begin{smallmatrix}a&b\\c&d \end{smallmatrix}\big)$ appearing in the expression of  $\gothU_p^\calB(\kappa)$, we have  $a \in p\cdot \ZZ_p^\times$, $ d \in \ZZ_p^\times$, $b, c \in \ZZ_p$, and $a d \equiv p \pmod {p^m}$.
Using Lemma~\ref{L:congruence} as well as the congruence $\frac ad \equiv \frac{p}{d^2} = p \omega^{-2}(d)\langle d\rangle^{-2} \pmod {p^3}$, we conclude (2) in the same way as above.
\end{proof}

We now proceed to prove Theorem~\ref{T:theorem B}.

\begin{notation}
Put $q = 1$ if $p=2$ and $q=\frac{p-1}{2}$ if $p >2$.

We write the characteristic series of $U_p$ acting on $S_\calB^{D, \dagger}(U; \kappa)$ as
\[
\Char(\gothU_p^\calB(\kappa), S_\calB^{D, \dagger}(U; \kappa)\big) = 1+ c_1(w) X + c_2(w) X^2 + \cdots \in 1+ \calO\langle w\rangle \llbracket X\rrbracket.
\]
\end{notation}

\begin{theorem}
\label{T:sharp Hodge bound}
Assume $m \geq 4$ as before.
We have the following results regarding the Newton polygon.
\begin{enumerate}
\item
 For any $w_0 \in \calW(x\psi_m; p^{-1})$, the Newton polygon of the power series $1+ c_1(w_0)X + \cdots $ lies above the polygon starting at $(0,0)$  with slopes given by
\begin{equation}
\label{E:slopes of improved Hodge polygon}
\bigcup_{n=0}^\infty \bigcup_{r=0}^ {q-1}\big\{
\alpha_0(\psi_m\omega^{-2r})+qn+r, \dots, \alpha_{t-1}(\psi_m\omega^{-2r})+qn+r\big\}.
\end{equation}

\item
For each $n \in \NN$, let $\lambda_n$ denote the sum of $n$ smallest numbers in \eqref{E:slopes of improved Hodge polygon}.
Then
\[
c_{kt}(w) \in p^{\lambda_{kt}} \cdot \calO\langle w\rangle^\times, \quad \textrm{for every }k \in \NN.
\]
\item
For any $w_0 \in \calW(x\psi_m; p^{-1})$, the  Newton polygon of the power series $1+ c_1(w_0)X + \cdots $ passes through the point $(kt, \lambda_{kt})$ for each $k \in \NN$ (which lies on the Hodge polygon in (1)).
In particular, the $n$th slope of this Newton polygon belongs to $\big[ \lfloor \frac nt \rfloor, \lfloor \frac nt \rfloor +1 \big]$.
\end{enumerate}
\end{theorem}
\begin{proof}
(1) Recall from Proposition~\ref{P:Tl Up overconvergent equiv classical}(2), the matrix for $U_p$ satisfies
\[
\gothU_p^\calB(\kappa) - \gothU_p^{\cl, \infty} \in \Err_p.
\]
We now uses the basis \[
e_0(\psi_m), \dots, e_{t-1}(\psi_m), pe_0(\psi_m\omega^{-2})z, \dots, pe_{t-1}(\psi_m \omega^{-2})z, p^2e_0(\psi_m\omega^{-4}) z^2, \dots
\]
As a result, we need to conjugate both $\gothU_p^\calB(\kappa)$ and $\gothU_p^{\cl, \infty}$ by an infinite block diagonal matrix whose block entries are $t\times t$-matrices in $\rmM_t(\calO)$. Thus, the action of $U_p$ is given by a new matrix $\gothU_p^{\calB, \bfe}$ which is congruent to
\[
\Diag\big( \gothU_p^{\cl, \bfe}(\psi_m),\ p\cdot  \gothU_p^{\cl, \bfe}(\psi_m \omega^{-2}),\ p^2\cdot  \gothU_p^{\cl, \bfe}(\psi_m \omega^{-4}),\ \dots \big)
\]
modulo $\Err_p$.
In particular, for $i=0, \dots, t-1$, the $((qn+r)t+i)$th row of $\gothU_p^{\calB, \bfe}$ is entirely divisible by $p^{\alpha_i(\psi_m \omega^{-2r}) + qn+r}$.
Therefore the Hodge polygon of $\gothU_p^{\calB, \bfe}$ lies above the Hodge polygon with slopes given by \eqref{E:slopes of improved Hodge polygon}. This improves the result of Theorem~\ref{T:weak Hodge polygon} (when $m \geq 4$).

(2) By the proof of (1), we know that $c_{kt}(w) \in p^{\lambda_{kt}} \cdot \calO\langle w\rangle$ for each $k \in \NN$.
It suffices to show that the reduction of $p^{-\lambda_{kt}} c_{kt}(w)$ modulo $\varpi$ lies in $\FF^\times \subset \FF[w]$.

Note that, if we think of $\gothU_p^{\calB, \bfe}$ as an infinite block matrix with $t \times t$-matrices as entries, then its $(i,j)$-block for $i>j$ is entirely divisible by $p^{i+2}$. So it will not contribute to the reduction of  $p^{-\lambda_{kt}} c_{kt}(w)$ modulo $\varpi$.
In other words, if $M_n$ denotes the $t\times t$-matrix appearing as the $(n,n)$-block entry of $\gothU_p^{\calB, \bfe}$, then
\[
p^{-\lambda_{kt}} c_{kt}(w)\ \equiv \ p^{-\lambda_{kt}} \prod_{n =0}^{k-1} \det (M_n) \pmod \varpi.
\]

Using the congruence relation discussed in (1) and Proposition~\ref{P:Tl Up overconvergent equiv classical}(2), we see that the diagonal $t\times t$-matrices are exactly given by
\[
\gothU_p^{\cl ,\bfe}(\psi_{m,w})  \textrm{ modulo }p^2, \quad p \cdot \big( \gothU_p^{\cl ,\bfe}(\psi_{m,w-2}\omega^{-2})  \textrm{ modulo }p^2\big),  \quad p^2 \cdot \big( \gothU_p^{\cl ,\bfe}(\psi_{m,w-4}\omega^{-4})  \textrm{ modulo }p^2\big), \dots
\]
Consequently, the reduction of $p^{-\lambda_{kt}} c_{kt}(w)$ modulo $\varpi$ is the same as the product
\[
\prod_{n=0}^{k-1} \det \big( \overline \gothU_p^{\cl, \bfe}(\psi_{m,w-2n} \omega^{-2n})\big) \textrm{ mod } \varpi.
\]
By Proposition~\ref{P:determine of Up deformed}, each factor belongs to $\FF^\times$ and so is the product.  (2) follows from this.

(3) Since (2) implies that the Newton polygon agrees with the Hodge polygon at points $(kt, \lambda_{kt})$ for all $k \geq 0$, the Newton polygon of the power series $1+ c_1(w_0) X+ \cdots $ is confined between the Hodge polygon of (1) and the polygon with vertices $(kt, \lambda_{kt})$.  (3) is immediate from this.
\end{proof}

Theorem~\ref{T:theorem B} is a corollary of Theorem~\ref{T:sharp Hodge bound} using the Jacquet--Langlands correspondence~\eqref{E:Jaquet-Langlands classical}.

\begin{remark}
\label{R:m=3}
Assume $p>2$.
When $m=3$, Theorem~\ref{T:sharp Hodge bound}(1) still holds.
But the argument in (2) fails in that, for example, there might be $p^2w$ terms in $(1,0)$-block entry for the matrix $\gothU_p^{\cl, \bfe}$. Apriori, they may have nontrivial contribution to the reduction of $p^{-\lambda_{kt}}a_{kt}(w)$ modulo $\varpi$.
So we can only conclude that the reduction is a unit in $\FF\llbracket w \rrbracket$ but not necessarily a unit in $\FF\langle w \rangle$.
The slope estimate would then only work over some \emph{open} disk of radius $p^{-1}$.
Nonetheless, we still expect our theorem continue to hold as long as $m \geq 2$.
It would be interesting to know how to extend our argument to the case $m=2,3$.
\end{remark}

\begin{corollary}
\label{C:precise NP computation}
Assume $m \geq 4$ as before. Let $\HP(\psi_m)$ (resp. $\NP(\psi_m)$) denote the Hodge polygon (resp. Newton polygon) of the $U_p$-action on $S_2^D(U; \psi_m)$; we write $\HP(\psi_m)(i)$ (resp. $\NP(\psi_m)$(i)) for the $y$-coordinate of the polygon when the $x$-coordinate is $i$.

Fix $r=0, \dots, q-1$.
Suppose that $(s_0, \NP(\psi_m\omega^{-2r})(s_0))$ is a \emph{vertex} of the Newton polygon $\NP(\psi_m\omega^{-2r})$ and suppose that
\begin{equation}
\label{E:NP<HP+1}
\NP(\psi_m\omega^{-2r})(s) < \HP(\psi_m\omega^{-2r})(s-1) + 1
\footnote{Note that the Newton polygon is evaluated at $s$ and the Hodge polygon is evaluated at $s-1$.} \textrm{ for all } s = 1, \dots, s_0.
\end{equation}
Then for any $s=0, \dots, s_0$, any $n \in \ZZ_{\geq 0}$, and any $w_0 \in \calW(x\psi_m; p^{-1})$, the $( qnt+ rt +s)$th slope of the power series $1+ c_1(w_0)X + \cdots$ is the $s$th $U_p$-slope on $S_2^D(U; \psi_m \omega^{-2r})$ plus $ qn + r $.
\end{corollary}
\begin{proof}
As in the proof of Theorem~\ref{T:sharp Hodge bound}(2),
$c_{qnt+ rt +s}(w_0)$ is divisible by
$
p^{\lambda_{qnt+ rt +s}}.
$
The approximation in the proof of Theorem~\ref{T:sharp Hodge bound}(1) also implies that, modulo $p^{\lambda_{qnt+ rt +s-1}} \cdot p$, this number is equal to
\[
\Big(
\prod_{a = 0}^{nq+r-1} p^a\det \gothU_p^{\cl, \bfe}(\psi_m \omega^{-2a}) \Big) \cdot p^{(nq+r)s} \cdot \ \big( \textrm{coefficient of $X^s$ in } \Char(U_p; S_2^D(\psi_m \omega^{-2r}))\, \big).
\]
So under the hypothesis \eqref{E:NP<HP+1}, this implies that, for each $s$,
\begin{itemize}
\item
if $s$ is a vertex of $\NP(\psi_m\omega^{-2r})$, the valuation of $c_{qnt+rt+s}(w_0)$ is determined by the Newton Polygon of the classical forms, i.e.
\[
v(c_{qnt+ rt +s}(w_0)) = v(c_{qnt+rt}(w_0)) + \NP(\psi_m\omega^{-2r})(s) + (qn+r)s,
\]

\item
if $s$ is not a vertex of $\NP(\psi_m\omega^{-2r})$, the valuation of $c_{qnt+rt+s}(w_0)$ is either determined by the Newton Polygon of the classical forms or greater than or equal to $\lambda_{qnt+rt+s-1}+1$; in either case, we have
\[
v(c_{qnt+ rt +s}(w_0)) \geq v(c_{qnt+rt}(w_0)) + \NP(\psi_m\omega^{-2r})(s) + (qn+r)s.
\]
\end{itemize}
Since $(s_0, \NP(\psi_m\omega^{-2r})(s_0))$ is a vertex, the $(qnt+ rt +s)$th slope, for $s=0, \dots, s_0$, of the power series $ 1+ c_1(w_0)X+ \cdots$ agrees with the $s$th slope of $\Char(U_p; S_2^D(\psi_m \omega^{-2r}))$ plus the normalizing factor $qn+r$.
\end{proof}

\begin{remark}
We emphasize that the sequence given by $s$th $U_p$-slope on $S^D_2(U; \psi_m \omega^{-2r})$ plus $qn+r$, as $n$ increases, is an arithmetic progression with common difference $q$ (\emph{but not $1$}).  This is due to the periodic appearance of the powers of the Teichm\"uller character.
This agrees with the computation of Kilford and McMurdy \cite{kilford, kilford-mcmurdy} in some special cases (with $m=2$), where the common difference is $2$ when $p=5$, and is $\frac 32$ (which can be further broken up into two arithmetic progressions with common difference $3$) when $p=7$.
\end{remark}

\begin{example}
\label{Ex:explicit slopes rare examples}
We provide an example to better understand the strength of \eqref{E:NP<HP+1}.
Consider the explicit example in Section~\ref{Section:explicit example} with $D = \QQ\langle \i, \j\rangle$ and $p=3$.  We first consider the $m=3$ case where we
 take $U$ to be
\begin{equation}
\label{E:p=3,m=3}
U= D^\times(\ZZ_2) \times \prod_{l \neq 2,3}\GL_2(\ZZ_l) \times \begin{pmatrix}
\ZZ_3^\times & \ZZ_3\\
27 \ZZ_3 & 1+ 3\ZZ_3
\end{pmatrix}
\end{equation}
and $\psi_3$ to be a character of $\ZZ_3^\times$ of conductor $27$.
Then $S_2^D(U; \psi_3)$ is $3$-dimensional, and the action of $U_3$ on the a basis is given by
\begin{equation}
\label{E:M3}
\begin{pmatrix}
\zeta_9 &\zeta_9^2& \zeta_9^8\\
\zeta_9^4 &\zeta_9^2& \zeta_9^5\\
\zeta_9^7&\zeta_9^2&\zeta_9^2
\end{pmatrix}.
\end{equation}
Its Newton polygon has slopes $\frac 16, \frac 12$, and $\frac 56$ and the Hodge polygon has slopes $0, \frac 12$, and $1$.

For the case $m=4$, we take $U$ to be as in \eqref{E:p=3,m=3} except the number $27$ is replaced by $81$. We take the character $\psi_4$ to have conductor $81$.
Then $S_2^D(U; \psi_4)$ is $9$-dimensional, and the action of $U_3$ on a basis is given by
\begin{equation}
\label{E:M4}
\begin{pmatrix}
\zeta^{19} &0&0&0&\zeta^2&\zeta^{17}&0&0&0\\
0&0&0&\zeta^{13}&0&0&0&\zeta^{20}&\zeta^{23}\\
0&\zeta^{11}&\zeta^2&0&0&0&\zeta^7&0&0\\
\zeta&0&0&0&\zeta^2&\zeta^8&0&0&0\\
0&0&0&\zeta^{22}&0&0&0&\zeta^{20}&\zeta^{14}\\
0&\zeta^{11}&\zeta^{20}&0&0&0&\zeta^{16}&0&0\\
\zeta^{10}&0&0&0&\zeta^2&\zeta^{26}&0&0&0\\
0&0&0&\zeta^4&0&0&0&\zeta^{20}&\zeta^5\\
0&\zeta^{11}&\zeta^{11}&0&0&0&\zeta^{25}&0&0
\end{pmatrix},
\end{equation}
where $\zeta$ is a primitive $27$th root of unity.  The Newton polygon of this matrix has slopes $\frac{1}{18}, \frac{1}{6}, \frac{5}{18}, \dots, \frac{17}{18}$, and the Hodge polygon has slopes $0, 0, 0, \frac 12, \frac 12, \frac 12, 1, 1, 1$.
In this case, the number $s_0$ in Corollary~\ref{C:precise NP computation} can be taken to be $6$; so we can determine about ``two thirds" of all slopes using Corollary~\ref{C:precise NP computation}.

\end{example}
We now return to the general case.

\begin{theorem}
\label{T:improved main theorem}
Assume $m \geq 4$ as before.
Let $\ord(\psi_m\omega^{-2n})$ denote the dimension of the ordinary part of $S_2^D(U; \psi_m \omega^{-2n})$, or equivalently, the multiplicity of slope $0$ in $\NP(\psi_m \omega^{-2n})$.

Then the spectral variety $\Spc_D \times_\calW \calW(x\psi_m; p^{-1})$ is a disjoint union of subvarieties
\[
X_0,\ X_{(0,1]},\ X_{(1,2]}, \ X_{(2,3]},\ \dots
\]
such that each subvariety is finite and flat over $\calW(x\psi_m; p^{-1})$, and  for any closed point $x \in X_{(n,n+1]}$ (resp. $x \in X_0$), we have $v(a_p(x)) \in (n,n+1]$ (resp. $v(a_p(x))=0$).
Moreover, the degree of $X_{(n,n+1]}$ over $\calW(x\psi_m; p^{-1})$ is exactly
\[
t + \ord(\psi_m \omega^{-2n-2}) - \ord(\psi_m \omega^{-2n}).
\]
In particular, this number depends only on $n \textrm{ mod } q$.
\end{theorem}
\begin{proof}
It suffices to show that, for a fixed $n \in \ZZ_{\geq 0}$ and any $w_0 \in \calW(x\psi_m; p^{-1})$, the number of slopes of $1+ c_1(w_0) X + \cdots $ less than or equal to $n$, is \emph{independent of $w_0$} and is equal to $nt + \ord(\psi_m\omega^{-2n})$.
If so, the subspace
\[
X_{[0,n]} = \big\{ (x, w_0) \in \Spc_D \times_\calW \calW(x\psi_m; p^{-1})\,|\, v(a_p(x)) \leq n \big\}
\]
is finite and flat of degree $nt + \ord(\psi_m\omega^{-2n})$ over $\calW(x\psi_m; p^{-1})$, and it would follow that $X_{[0,n]}$ is both open (by definition) and closed (by finiteness) in $\Spc_D \times_\calW \calW(x\psi_m; p^{-1})$, and hence a union of connected components.  The theorem would then follow from this.

To estimate the number of slopes less than or equal to $n$, we use the Hodge polygon lower bound in Theorem~\ref{T:sharp Hodge bound}.  It then suffices to prove that
\begin{equation}
\label{E:ordinary condition}
\begin{split}
&v(c_{nt + \ord(\psi_m \omega^{-2n})}(w_0)) = \lambda_{nt} + n \cdot \ord(\psi_m \omega^{-2n}), \textrm{ and }\\
&\qquad v(c_{nt + s}(w_0)) > \lambda_{nt} + ns \textrm{ for }s >\ord(\psi_m \omega^{-2n}).
\end{split}
\end{equation}
We again go back to the slope estimate in the proof of Theorem~\ref{T:sharp Hodge bound} (like in the proof of Corollary~\ref{C:precise NP computation}). It is easy to deduce that $c_{nt+s}(w_0)$ for $s \geq \ord(\psi_m \omega^{-2n})$ is congruent to
\[
\big(
\prod_{i = 0}^{n-1} \det \gothU_p^{\cl, \bfe}(\psi_m \omega^{-2i}) \big) \cdot p^{ns} \cdot \ \big( \textrm{coefficient of $X^s$ in } \Char(U_p; S_2^D(\psi_m \omega^{-2n}))\, \big)
\]
modulo $p^{\lambda_{nt} + ns+1}$.
The valuation inequalities \eqref{E:ordinary condition} follow from this congruence relation.
\end{proof}

\begin{remark}
We certainly expect that $X_{(i,i+1]}$ is the disjoint union of $X_{(i,i+1)} \coprod X_{i+1}$ (with the obvious meaning); but we do not know how to prove this because, apriori, the error terms from $w$ might present an obstruction.
\end{remark}

\begin{remark}
\label{R:global decomposition under condition}
Using Corollary~\ref{C:precise NP computation} and the argument above, we can show that, when there is a vertex $(s_0, \NP(\psi_m \omega^{-2r})(s_0))$ of the Newton polygon $\NP(\psi_m \omega^{-2r})$ as in Corollary~\ref{C:precise NP computation}, we can get a further decomposition of $X_{(qn+r, qn+r+1]}$ separating those points whose $a_p$-slopes are first $s_0$ $U_p$-slopes on $S_2^D(U; \psi_m \omega^{-2r})$ plus $qn+r$.
\end{remark}

\section{Techniques for separation by residual pseudo-representations}
\label{Section:separation residue}

We motivate this section by pointing out that the power of Corollary~\ref{C:precise NP computation} is largely determined by how close the Newton polygon is to the Hodge polygon.  The application of this result will be increasingly limited when the level subgroup $U$ gets smaller.
An natural idea to loosen the condition \eqref{E:NP<HP+1} is to separate the space of automorphic forms using  the \emph{tame} Hecke algebras.

In fact, we will show that one can obtain a natural direct sum decomposition of the space of overconvergent automorphic forms according to the residual Galois pseudo-representations attached.  Furthermore, we can reproduce main theorems of the previous section for each direct summand.
We also emphasize that this decomposition should have its own interest.

We keep the notation as in the previous section.  In particular, we assume Hypothesis~\ref{H:psi conductor pm}: $m \geq 4$.

\subsection{Pseudo-representations}
Let $G_{\QQ,\calS}$ denote the Galois group of the maximal extension of $\QQ$ unramified outside $\calS$ (see Subsection~\ref{S:setup for D} for $\calS$).
Let $R$ be a (topological) ring.
A ($2$-dimensional) \emph{pseudo-representation} is a (continuous) map $\rho: G_{\QQ, \calS} \to R$ such that, for $g_1, g_2, g_3 \in G_{\QQ,\calS}$, we have $\rho(1) = 2$, $\rho(g_1g_2) = \rho(g_2g_1)$, and
\[
\rho(g_1) \rho(g_2) \rho(g_3) + \rho(g_1g_2g_3) + \rho(g_1g_3g_2) = \rho(g_1) \rho(g_2g_3)+ \rho(g_2) \rho(g_1g_3) + \rho(g_3) \rho(g_1g_2).
\]
Let $\rho: G_{\QQ, \calS} \to \calO$ be a pseudo-representation.
\begin{itemize}
\item
If $\chi: G_{\QQ, \calS} \to \calO^\times$ is a continuous character,  then $(\rho\otimes \chi)(g): = \rho(g) \chi(g)$ is  a pseudo-representation.
\item
We use $\bar \rho: G_{\QQ,\calS} \to \FF$ to denote the reduction $\bar \rho(g) : = \rho (g) \textrm{ mod }\varpi$; it is called the \emph{residual pseudo-representation} associated to $\rho$.
\item
The (residual) pseudo-representation is uniquely determined by the its evaluation on the geometric Frobenius: $\rho(\Frob_l)$ for $l \notin \calS$.
\end{itemize}

It is known that to each automorphic representation $\pi$ appearing in $S_2^D(U; \psi_m)$,
there exists a pseudo-representation $\rho_\pi: G_{\QQ, \calS} \to \calO$ such that $\rho(\Frob_l) = a_l(\pi)$ for all $l \notin \calS$.
We say that a residual pseudo-representation $\bar \rho: G_{\QQ, \calS} \to \FF$ \emph{appears} in a space of automorphic forms $S_2^D(U; \psi_m)$ if there is an automorphic representation $\pi$ appearing in $S_2^D(U; \psi_m)$ such that the reduction of the associated pseudo-representation is $\bar \rho$.

The goal of this section is to decompose the space $S_\calB^{D, \dagger}(U; \kappa)$ according to the residual pseudo-representations appearing in the space of \emph{weight two} classical automorphic forms.\footnote{It should not be too surprise to see that we only need weight two modular forms, as it was already observed by Serre \cite{serre} that all modular residual pseudo-representations appear in weight two.}
The key is to use the tame Hecke action to break up the space $S_\calB^{D, \dagger}(U; \kappa)$.
We start with the decomposition over the space of classical automorphic forms.

\begin{notation}
We use $\scrB(U; \psi_m)$ to denote all residual pseudo-representations $\bar \rho$ that appear in $S^\cl: = \bigoplus_{r=0}^{q-1}
S_2^D(U; \psi_m \omega^{-2r}; \omega^r)$.
For each pair of distinct residual pseudo-representations $\bar \rho, \bar \rho' \in \scrB(U; \psi)$, we pick a prime $l_{\bar \rho, \bar \rho'} \notin \calS$ such that $\bar \rho(\Frob_{l_{\bar \rho, \bar \rho'}}) \neq \bar \rho'(\Frob_{l_{\bar \rho, \bar \rho'}})$. We fix a lift $\tilde a_{l_{\bar \rho, \bar \rho'}}(\bar \rho) \in \calO$ of $\bar \rho(\Frob_{l_{\bar \rho, \bar \rho'}})$ and a lift $\tilde a_{l_{\bar \rho, \bar \rho'}}(\bar \rho') \in \calO$ of $\bar \rho'(\Frob_{l_{\bar \rho, \bar \rho'}})$.

For $\bar \rho \in \scrB (U; \psi_m)$, consider the following tame Hecke operator
\[
P_{\bar \rho} : = \prod_{\bar \rho' \neq \bar \rho}
\big( T_{l_{\bar \rho, \bar\rho'}} - \tilde a_{l_{\bar \rho, \bar \rho'}}(\bar \rho')\big) \big/ \big( \tilde a_{l_{\bar \rho, \bar \rho'}}(\bar \rho) - \tilde a_{l_{\bar \rho, \bar \rho'}}(\bar \rho')\big).
\]
Note that $P_{\bar \rho}$ defines an endomorphism of the integral model $
S_2^D(U; \psi_m \omega^{-2r}; \omega^r; \calO)$ for each $r$.
The operator $P_{\bar \rho}$ depends on the choice of the lifts $\tilde a_{\bar \rho, \bar \rho'}(\bar \rho)$ and $\tilde a_{\bar \rho, \bar \rho'}(\bar \rho')$'s.
\end{notation}

\begin{lemma}
\label{L:decomposition classical forms}
Fix $r \in\{0, \dots, q-1\}$.
Let $P_{\bar \rho,r} $ denote the action of $P_{\bar \rho}$ on  the space of classical automorphic forms $
S_2^D(U; \psi_m \omega^{-2r}; \omega^r; \calO)$.
Then $P_{\bar \rho, r}^2 \equiv P_{\bar \rho, r} \pmod {\varpi}$.
The limit
\[
\widetilde P_{\bar \rho, r}^\cl: = \lim_{n \to \infty} (P_{\bar \rho, r})^{p^n}
\]
exists and it is the projection to the direct sum $V(\bar\rho)_r$ of subspaces $V(\pi)$ over all automorphic representations $\pi$ appearing in $S_2^D(U; \psi_m \omega^{-2r}; \omega^r)$ for which the associated pseudo-representation reduces to $\bar \rho$.
In particular, we have
\[
\big(\widetilde P_{\bar \rho, r} \big)^2 = \widetilde P_{\bar \rho, r}, \quad
\widetilde P_{\bar \rho, r} \widetilde P_{\bar \rho', r} = 0 \textrm{ if }\bar \rho \neq \bar \rho',\ \textrm{ and }
\sum_{\bar \rho \in \scrB(U; \psi_m)} \widetilde P_{\bar \rho, r} = \mathrm{id}.
\]
Moreover, the definition of $\widetilde P_{\bar \rho, r}$ is independent of the choice of the lifts $\tilde a_{l_{\bar \rho, \bar \rho'}}(\bar \rho)$ and $\tilde a_{l_{\bar \rho, \bar \rho'}}(\bar \rho')$'s; and it defines a direct sum decomposition of the integral model
\[
S_2^D(U; \psi_m \omega^{-2r}; \omega^r; \calO) \cong \bigoplus_{\bar\rho \in \scrB(U; \psi_m)} V(\bar \rho; \calO)_r.
\]
\end{lemma}
\begin{proof}
Note that, $P_{\bar \rho, r}$ acts on each $V(\pi)$ by some element in $(\varpi)$ if $\bar \rho_\pi \neq \bar \rho $, and by some $1$-unit if $\bar \rho_\pi = \bar \rho$.  The Lemma follows from this immediately.
\end{proof}

The upshot is that one can extend the decomposition above to the case of overconvergent automorphic forms.

\subsection{Some infinite matrices}
For each $r$, we identify $S^D_2(U; \psi_m \omega^{-2r}; \omega^r)$ with  $\oplus_{i=0}^{t-1} E$ by evaluating the automorphic forms at $\gamma_0,\gamma_1, \dots, \gamma_{t-1}$.
This way, the operators $P_{\bar \rho,r}$ and $\widetilde P_{\bar \rho,r}$ are represented by two  $t\times t$-matrices $\gothP_{\bar \rho,r}^\cl, \widetilde \gothP_{\bar \rho,r}^\cl \in \rmM_{t}(\calO)$.

We use $\gothP_{\bar \rho}^{\cl, \infty}$ (resp. $\widetilde \gothP_{\bar \rho}^{\cl, \infty}$) to denote the infinite block diagonal matrix
whose diagonal block-entries are $\gothP_{\bar \rho,0}^\cl, \gothP_{\bar \rho,1}^\cl, \dots$ (resp. $\widetilde \gothP_{\bar \rho,0}^\cl, \widetilde \gothP_{\bar \rho,1}^\cl, \dots$).

Note that $P_{\bar \rho}$ only involves Hecke operators, so it also acts on the space of overconvergent automorphic forms $S^{D,\dagger}_{ \calB}(U;\kappa)$, where $\kappa$ is the continuous character of $\ZZ_p^\times$ with values in $(A^\circ)^\times$ as defined in Notation~\ref{N:character kappa}.
Let $\gothP_{\bar \rho}^\calB(\kappa)$ denote the matrix for $P_{\bar \rho}$ under the basis  given by $1_0, \dots, 1_{t-1}, pz_0, \dots, pz_{t-1}, p^2 z^2_0, \dots$ as in Notation~\ref{N:matrix for Tl and Up}.

By Proposition~\ref{P:Tl Up overconvergent equiv classical}(1), we have that
\begin{equation}
\label{E:congruence Q with P}
\gothP_{\bar \rho}^\calB(\kappa)
 \equiv  \gothP_{\bar \rho}^{\cl, \infty} \textrm{ modulo the error space }\mathbf{Err} \textrm{ in }\eqref{E:error matrix}.
\end{equation}
The next Proposition says that we can improve the infinite matrix $\gothP_{\bar \rho}^\calB(\kappa)$ into a projection, as we did above so that we can factor out the subspace of overconvergent automorphic forms
corresponding to the Galois pseudo-representation $\bar \rho$.


\begin{proposition}
\label{P:decomposition}
Keep the notation as above.
\begin{enumerate}
\item We have $(\widetilde \gothP_{\bar \rho}^{\cl, \infty})^2 = \widetilde\gothP_{\bar \rho}^{\cl, \infty}$, $\widetilde\gothP_{\bar \rho}^{\cl, \infty}  \widetilde\gothP_{{\bar \rho}'}^{\cl, \infty} = 0$ if ${\bar \rho} \neq {\bar \rho}'$, and $\sum_{{\bar \rho} \in \scrB(U; \psi_m)} \widetilde\gothP_{\bar \rho}^{\cl, \infty} = I_\infty$, where $I_\infty: = \Diag(1)$ denotes the infinite identity matrix.
\item
The limit
\[
\widetilde \gothP_{\bar \rho}^\calB(\kappa): = \lim_{n \to \infty} (\gothP_{\bar \rho}^\calB(\kappa))^{p^n}
\]
exists.  Moreover, we have
\[
\big(\widetilde \gothP_{\bar \rho}^\calB(\kappa)\big)^2 = \widetilde \gothP_{\bar \rho}^\calB(\kappa), \
\widetilde \gothP_{\bar \rho}^\calB(\kappa) \widetilde \gothP_{{\bar \rho}'}^\calB(\kappa) = 0 \textrm{ for }{\bar \rho} \neq {\bar \rho}', \textrm{ and }
\sum_{{\bar \rho} \in \scrB(U; \psi_m)} \widetilde \gothP_{\bar \rho}^\calB(\kappa) = I_\infty.
\]

\item
We have a decomposition of Banach $A$-modules respecting the $U_p$-action:
\[
S^{D,\dagger}_{\calB}(U; \kappa) = \bigoplus_{{\bar \rho} \in \scrB(U; \psi_m)}\widetilde \gothP_{\bar \rho}^\calB(\kappa) S^{D,\dagger}_{\calB}(U; \kappa).
\]
Consequently, we have a product formula for the characteristic power series
\[
\Char(U_p; S^{D,\dagger}_{\calB}(U; \kappa)) = \prod_{{\bar \rho} \in \scrB(U; \psi_m)} \Char \big(U_p; \widetilde \gothP_{\bar \rho}^\calB(\kappa) S^{D,\dagger}_{\calB}(U; \kappa)  \big).
\]
\item
We have the following congruence relation: for every $\bar \rho\in \scrB(U; \psi_m)$, the difference of the infinite matrices $ \widetilde \gothP_{\bar \rho}^\calB(\kappa) -  \widetilde \gothP_{\bar \rho}^{\cl, \infty}$ belongs to the space $\mathbf{Err}$ in \eqref{E:error matrix}.
\end{enumerate}
\end{proposition}
\begin{proof}
(1) follows from the corresponding properties of $\widetilde\gothP_{\bar \rho,r}^\cl$ in Lemma~\ref{L:decomposition classical forms}.

For (2), we observe that $\gothP_{\bar \rho}^\calB(\kappa) \equiv \gothP_{\bar \rho}^{\cl, \infty} \pmod \varpi$ by Proposition~\ref{P:Tl Up overconvergent equiv classical}(1).  So by Lemma~\ref{L:decomposition classical forms},
\begin{equation}
\label{E:P idempotent mod pi}
\big(\gothP_{\bar \rho}^\calB(\kappa)\big)^2
\equiv
\gothP_{\bar \rho}^\calB(\kappa) \pmod \varpi.
\end{equation}

Easy induction proves that $(\gothP_{\bar \rho}^\calB(\kappa))^{p^{i+1}} \equiv (\gothP_{\bar \rho}^\calB(\kappa))^{p^i} \pmod{\varpi^i}$. So the limit $\widetilde \gothP_{\bar \rho}^\calB(\kappa): = \lim_{i \to \infty} (\gothP_{\bar \rho}^\calB(\kappa))^{p^i}$ exists.
The property $\big(\widetilde \gothP_{\bar \rho}^\calB(\kappa)\big)^2 = \widetilde \gothP_{\bar \rho}^\calB(\kappa)$ also follows from \eqref{E:P idempotent mod pi}.

Now for two pseudo-representations $\bar \rho \neq \bar \rho'$ in $\scrB(U; \psi_m)$,
we have
\[
\gothP_{\bar \rho}^\calB(\kappa) \gothP_{{\bar \rho}'}^\calB(\kappa) \equiv \gothP_{\bar \rho}^{\cl, \infty} \gothP_{{\bar \rho}'}^{\cl, \infty} \equiv 0 \pmod \varpi.
\]
It then follows that  $\widetilde \gothP_{\bar \rho}^\calB(\kappa) \widetilde \gothP_{{\bar \rho}'}^\calB(\kappa) = 0$ (note that it is important to know that $\gothP_{\bar \rho}^\calB(\kappa)$ commutes with $ \gothP_{{\bar \rho}'}^\calB(\kappa)$ because both operators can be expressed in terms of Hecke operators.)

Similarly, if we start with
\[
\sum_{{\bar \rho} \in \scrB(U; \psi_m)} \gothP_{\bar \rho}^\calB (\kappa)\equiv \sum_{{\bar \rho} \in \scrB(U; \psi_m)} \gothP_{\bar \rho}^{\cl, \infty} \equiv  I_\infty \pmod \varpi,
\]
raising it to $p^i$th power implies that
\[
I_\infty \equiv
\sum_{{\bar \rho} \in \scrB(U; \psi_m)} \big(\gothP_{\bar \rho}^\calB(\kappa)\big)^{p^i} \pmod {\varpi^i}.
\]
Here we used the fact that $\gothP_{\bar \rho}^\calB(\kappa) \gothP_{{\bar \rho}'}^\calB(\kappa) \equiv 0 \pmod \varpi$ for ${\bar \rho} \neq {\bar \rho}'$ and once again the crucial commutativity of $\gothP_{\bar \rho}^\calB(\kappa)$'s.
Taking the limit as $i \to \infty$ shows that $\sum_{{\bar \rho} \in \scrB(U; \psi_m)} \widetilde \gothP_{\bar \rho}^\calB(\kappa)  =I_\infty$.

(3) follows from (2) and the fact that $U_p$ commutes with each $\widetilde \gothP_{\bar \rho}^\calB(\kappa)$, as this operator is a limit of polynomials in tame Hecke operators.

We now check (4).
First recall some basic properties of the error space $\Err$ defined in \eqref{E:error matrix}.
For $M_1, M_2 \in \Err$, it is easy to see that $M_1M_2 \in \Err$ and $ \gothP_{\bar \rho}^{\cl, \infty} M_1, M_1 \gothP_{\bar \rho}^{\cl, \infty} \in \Err$.
Thus
\begin{equation}
\label{E:gothQ minus overline gothQ}
(\gothP_{\bar \rho}^\calB(\kappa))^{p^n} - (\gothP_{\bar \rho}^{\cl, \infty})^{p^n} =
\big(  \gothP_{\bar \rho}^{\cl, \infty} + (\gothP_{\bar \rho}^\calB(\kappa) - \gothP_{\bar \rho}^{\cl, \infty})\big)^{p^n} - ( \gothP_{\bar \rho}^{\cl, \infty})^{p^n} \in \Err
\end{equation}
because $\gothP_{\bar \rho}^\calB(\kappa) - \gothP_{\bar \rho}^{\cl, \infty} \in \Err$ by \eqref{E:congruence Q with P}.  Taking the limit as $n \to \infty$ proves (4).
\end{proof}

\begin{caution}
It is important to point out that, in \eqref{E:gothQ minus overline gothQ}, since $\gothP_{\bar \rho}^\calB(\kappa)$ and $ \gothP_{\bar \rho}^{\cl, \infty}$ do not commute with each other, we \emph{cannot} use binomial expansion formula to improve the congruence \eqref{E:gothQ minus overline gothQ}. Hence the limit $\widetilde \gothP_{\bar \rho}^\calB(\kappa)$ is \emph{not} a block diagonal matrix.
So Proposition~\ref{P:decomposition}(4) is the best congruence we could hope for.
\end{caution}

\begin{remark}
\label{R:pseudo-decomposition formal}
We should point out that decomposing a Banach Hecke module according to pseudo-Galois representations $\bar \rho$ is a quite formal process and can be done in a much greater generality.
However, it is often difficult to control the factor corresponding to each $\bar \rho$.
The advantage of our situation is that we can give a good ``model" of the factor corresponding to each $\bar \rho$.
\end{remark}

\subsection{${\bar \rho}$-part of classical automorphic forms}

Recall from Lemma~\ref{L:decomposition classical forms} that the space of classical automorphic forms $S_2^D(U; \psi_m\omega^{-2r}; \omega^r; \calO)$ for each $r$ is written as the direct sum $\bigoplus_{\bar \rho \in \scrB(U; \psi_m)} V(\bar \rho, \calO)_r$.
We put $V(\bar \rho)_r = V(\bar \rho; \calO)[\frac 1p]$ and $d_{\bar \rho, r}: = \dim V(\bar \rho)_r$.

Note that the operator $U_p$
acts on each $V(\bar \rho, \calO)_r$.
By Corollary~\ref{C:Hodge polygon less eq 1} (and Subsection~\ref{S:Hodge v.s. Newton}(6)), the  Hodge slopes $\alpha_0(\bar\rho)_r\leq \cdots \leq \alpha_{ d_{\bar \rho, r}-1}(\bar \rho)_r$ of the $U_p$-action on each $V(\bar \rho, \calO)_r$ belong to $[0,1]$. The same holds for the Newton slopes.
We pick a basis $e_0(\bar \rho)_r, \dots, e_{d_{\bar \rho, r}-1}(\bar \rho)_r$ of $V(\bar \rho, \calO)_r$ such that, the corresponding matrix $\gothU_p^{\cl, \bar \rho, r}$ of the $U_p$-action has $i$th row divisible by $p^{\alpha_i(\bar \rho)_r}$.


Providing $S_2^D(U; \psi_m \omega^{-2r}; \omega^r)$ with the natural basis of evaluation at $\gamma_0, \dots, \gamma_{t-1}$, and each $V(\bar \rho)_r$ with the basis above,
we write $\gothC_{\bar \rho ,r}$ and $\gothD_{\bar \rho, r}$ for the matrices for the natural inclusion and the natural projection $\widetilde \gothP_{\bar \rho, r}^{\cl}$:
\[
\xymatrix@C=30pt{
V(\bar \rho)_r \ar[r]^-{ \gothC_{\bar \rho,r}} & S_2^D(U; \psi_m \omega^{-2r}; \omega^r)
\ar[r]^-{\gothD_{\bar \rho, r}} & V(\bar \rho)_r.
}
\]
So $\gothC_{\bar \rho, r}$ is a $t \times d_{\bar \rho, r}$-matrix and $\gothD_{\bar \rho, r}$ is a $d_{\bar \rho, r} \times t$-matrix
such that
$ \gothC_{\bar \rho, r} \gothD_{\bar \rho, r} = \widetilde \gothP_{\bar \rho, r}^\cl$ and $\gothD_{\bar \rho, r} \gothC_{\bar \rho, r}  = I_{d_{\bar \rho, r}}$.

We point out that the number $d_{\bar \rho, r}$ and the matrices $\gothU_p^{\cl, \bar \rho, r}$, $\gothC_{\bar \rho, r}$, $\gothD_{\bar \rho, r}$, and $\widetilde\gothP^\cl_{\bar \rho, r}$ only depends on $r$ modulo $q$ as opposed to $r$.
\subsection{A model for the $\bar\rho$-part of overconvergent automorphic forms}
Proposition~\ref{P:decomposition} allows us to reduce the study of the $U_p$-action on $ S^{D,\dagger}_{\calB}(U; \kappa)$ to the $U_p$-action on each subspace $\widetilde \gothP_{\bar \rho}^\calB(\kappa) S^{D,\dagger}_{\calB}(U; \kappa)$, which we call the \emph{${\bar \rho}$-part of $S^{D,\dagger}_{\calB}(U; \kappa)$}.
This space is too abstract to study as pointed out in Remark~\ref{R:pseudo-decomposition formal}. We need to give it a ``model": $V({\bar \rho})^\infty_A$.

We set
\[
S^{\cl, \infty, \circ}_A: = \widehat \bigoplus_{r \geq 0} S^D_2(U; \psi_m \omega^{-2r}; \omega^r; \calO) \otimes_{\calO} A^\circ, \textrm{ and } S^{\cl, \infty}_A : = S^{\cl, \infty, \circ}_A \otimes_{\calO} E.
\]
We define the $U_p$-action on this space to be $\bigoplus_{r \geq 0} p^r \cdot U_p$.
Let $\widetilde \gothU_p^{\cl, \infty}$ denote the matrix for this action with respect the standard basis given by evaluation at $\gamma_0, \dots, \gamma_{t-1}$ of each of the summand. This matrix is the infinite block diagonal matrix whose diagonal components are $p^r \cdot \gothU_p^{\cl}(\psi_m\omega^{-2r})$.

We put
\[
V(\bar\rho)^{\infty, \circ}_A: = \widehat \bigoplus_{r \geq 0} V(\bar\rho; \calO)_r \otimes_{\calO} A^\circ, \textrm{ and }V(\bar\rho)^{\infty}_A  = V(\bar\rho)^{\infty, \circ}_A \otimes_{\calO} E.
\]
We define the $U_p$-action on this space to be $\bigoplus_{r \geq 0} p^r \cdot U_p$. The corresponding matrix with respect to the chosen basis on each $V(\bar \rho; \calO)_r$ is an infinite block diagonal matrix $\gothU_p^{\bar \rho, \infty}$ whose diagonal components are $p^r \cdot \gothU_p^{\cl, \bar \rho, r}$.

We write
\[
\gothC_{\bar \rho}^\infty: = \widehat \oplus_{r \geq 0} \gothC_{\bar \rho, r}: V(\bar \rho)_A^\infty \to S_A^{\cl, \infty} \quad \textrm{ and }\quad \gothD_{\bar \rho}^\infty: = \widehat \oplus_{r \geq 0} \gothD_{\bar \rho, r}: S_A^{\cl, \infty} \to V(\bar \rho)_A^\infty
\]
for the natural inclusion and projection, respectively.  So we have $\gothD_{\bar \rho}^\infty\gothC_{\bar \rho}^\infty = I_\infty$, and $\widetilde \gothP_{\bar \rho}^{\cl, \infty} = \gothC_{\bar \rho}^\infty \gothD_{\bar \rho}^\infty$  is the infinite block diagonal matrix composed of $\widetilde \gothP_{\bar \rho,r}^\cl$.

At the infinite level, we consider the following identification
\begin{equation}
\label{E:total space}
 S^{D,\dagger}_{\calB}(U; \kappa)=
\bigoplus_{i=0}^{t-1}  E\langle w, pz\rangle  =
\bigoplus_{i=0}^{t-1} \widehat \bigoplus_{n\geq 0}   E\langle w\rangle (pz)^{n}  \cong S_A^{\cl, \infty},
\end{equation}
where the first and the last equality are given by evaluation at the elements $\gamma_0,\gamma_1, \dots, \gamma_{t-1}$. This isomorphism does not respect the actions of the Hecke operators literally but we will show later that it approximately does.

\begin{proposition}
\label{P:identification of V(pi) and PS}
The following two natural morphisms are isomorphisms
\[
\xymatrix@C=40pt{
\varphi_{\bar \rho}:\ \widetilde \gothP_{\bar \rho}^\calB(\kappa) S^{D,\dagger}_{\calB}(U; \kappa) \subseteq
S^{D,\dagger}_{\calB}(U; \kappa)\ar[r]^-{\eqref{E:total space}}_-\cong & S_A^{\cl, \infty}
 \ar[r]^-{\gothD_{\bar \rho}^{ \infty}} &V({\bar \rho})^\infty_A;
}
\]
\[
\xymatrix@C=20pt{
\psi_{\bar \rho}:\ V({\bar \rho})^\infty_A \ar[r]^-{\gothC_{\bar \rho}^\infty} &  S_A^{\cl, \infty} \ar[rr]^-{\eqref{E:total space}^{-1}}_-\cong & &
S^{D,\dagger}_{\calB}(U; \kappa)
 \ar[r]^-{\widetilde \gothP_{\bar \rho}^{\calB}(\kappa)} &\widetilde \gothP_{\bar \rho}^\calB(\kappa) S^{D,\dagger}_{\calB}(U; \kappa) .
 }
\]
Moreover, $\psi_{\bar \rho}^{-1} =(1+ \boldsymbol \epsilon) \circ \varphi_{\bar \rho}$ for some endomorphism $\boldsymbol \epsilon: V({\bar \rho})^\infty_A \to V({\bar \rho})^\infty_A$ which, under the basis $\{ e_j (\bar \rho)_r\, |\, j = 0, \dots, d_{\bar \rho, r}-1 \textrm{ and }r \geq 0\}$, is an infinite matrix in
\[
\Err_{\bar \rho}: = \begin{pmatrix}
p\rmM_{d_{\bar \rho, 0}}(A^\circ) & p\rmM_{d_{\bar \rho, 0}\times d_{\bar \rho, 1}}(A^\circ) & p^2 \rmM_{d_{\bar \rho, 0}\times d_{\bar \rho, 2}}(A^\circ) & p^3\rmM_{d_{\bar \rho, 0}\times d_{\bar \rho, 3}}(A^\circ) &\cdots
\\
p^3\rmM_{d_{\bar \rho, 1}\times d_{\bar \rho, 0}}(A^\circ) &  p\rmM_{d_{\bar \rho, 1}}(A^\circ) & p\rmM_{d_{\bar \rho, 1}\times d_{\bar \rho, 2}}(A^\circ) &p^2 \rmM_{d_{\bar \rho, 1}\times d_{\bar \rho, 3}}(A^\circ) &  \cdots
\\
p^4  \rmM_{d_{\bar \rho, 2}\times d_{\bar \rho, 0}}(A^\circ) &p^3\rmM_{d_{\bar \rho, 2}\times d_{\bar \rho, 1}}(A^\circ) & p \rmM_{d_{\bar \rho, 2}\times d_{\bar \rho, 2}}(A^\circ)& p\rmM_{d_{\bar \rho, 2}\times d_{\bar \rho, 3}}(A^\circ) & \cdots\\
p^5\rmM_{d_{\bar \rho, 3}\times  d_{\bar \rho, 0}}(A^\circ) &p^4\rmM_{d_{\bar \rho, 3}\times  d_{\bar \rho, 1}}(A^\circ) &p^3\rmM_{d_{\bar \rho, 3}\times d_{\bar \rho, 2}}(A^\circ) &  p\rmM_{d_{\bar \rho, 3}}(A^\circ)&  \cdots
\\
\vdots & \vdots & \vdots & \vdots & \ddots
\end{pmatrix},
\]
where the $(i,j)$-block entry is
\begin{itemize}
\item
$ p\rmM_{d_{\bar \rho,i}}(A^\circ)$ if $i=j$,
\item
$p^{i-j+2}\rmM_{d_{\bar \rho, i}\times d_{\bar \rho, j}}(A^\circ)$ if $i > j$, and
\item
$p^{j-i}\rmM_{d_{\bar \rho, i}\times d_{\bar \rho, j}}(A^\circ)$ if $i < j$.
\end{itemize}
\end{proposition}
\begin{proof}
We first take the composition
\begin{align}
\label{E:varphi psi almost id}
\varphi_{\bar \rho} \circ \psi_{\bar \rho} - I_\infty &=  \gothD_{\bar \rho}^\infty \widetilde \gothP_{\bar \rho}^\calB(\kappa) \gothC_{\bar \rho}^\infty - I_\infty
\\
\nonumber&
= \gothD_{\bar \rho}^\infty \widetilde \gothP_\pi^{\cl, \infty} \gothC_{\bar \rho}^\infty - I_\infty + \gothD_{\bar \rho}^\infty \big(\widetilde \gothP_{\bar \rho}^\calB(\kappa) - \widetilde \gothP_{\bar \rho}^{\cl, \infty} \big) \gothC_{\bar \rho}^\infty.
\end{align}
Note that $\gothD_{\bar \rho}^\infty \widetilde \gothP_{\bar \rho}^{\cl, \infty} \gothC_{\bar \rho}^\infty - I_\infty = \gothD_{\bar \rho}^\infty \gothC_{\bar \rho}^\infty  \gothD_{\bar \rho}^\infty \gothC_{\bar \rho}^\infty - I_\infty = 0$ and
\[
\gothD_{\bar \rho}^\infty \big(\widetilde \gothP_{\bar \rho}^\calB(\kappa) - \widetilde \gothP_{\bar \rho}^{\cl, \infty} \big) \gothC_{\bar \rho}^\infty  \in \gothD_{\bar \rho}^\infty \cdot \Err \cdot \gothC_{\bar \rho}^\infty \subseteq \Err_{\bar \rho},
\]
where the last inclusion uses the fact that $\gothC_{\bar \rho}^\infty$ and $\gothD_{\bar \rho}^\infty$ are block diagonal matrices (but not with square blocks though).
Since all matrices in $I_\infty+ \Err_{\bar \rho}$ are invertible, $\varphi_{\bar \rho} \circ \psi_{\bar \rho}$ is an isomorphism.
Thus it suffices to prove that $\psi_{\bar \rho}$ is surjective.

For this, we need only to show the surjectivity of $\psi_{\bar \rho} \circ \gothD_{\bar \rho}^\infty$.  Note that
\[
\psi_{\bar \rho} \circ \gothD_{\bar \rho}^\infty = \widetilde \gothP_{\bar \rho}^\calB(\kappa)\gothC_{\bar \rho}^\infty \gothD_{\bar \rho}^\infty = \widetilde \gothP_{\bar \rho}^\calB(\kappa) \widetilde \gothP_{\bar \rho}^{\cl, \infty} = \widetilde \gothP_{\bar \rho}^\calB(\kappa) \big( I_\infty + (\widetilde \gothP_{\bar \rho}^{\cl, \infty} - \widetilde \gothP_{\bar \rho}^\calB(\kappa)) \big).
\]
By Proposition~\ref{P:decomposition}(4), the operator $I_\infty + (\widetilde \gothP_{\bar \rho}^{\cl, \infty} - \widetilde \gothP_{\bar \rho}^\calB(\kappa)) \in I_\infty+\Err$ is an isomorphism.  Then the surjectivity of $\psi_{\bar \rho} \circ \gothD_{\bar \rho}^\infty $ follows from the surjectivity of $\widetilde \gothP_{\bar \rho}^\calB(\kappa)$ onto $\widetilde \gothP_{\bar \rho}^\calB(\kappa) S^{D,\dagger}_{\calB}(U; \kappa)$.
This then concludes the proof of both $\varphi_{\bar \rho}$ and $\psi_{\bar \rho}$ being isomorphisms.

Finally, we observe that
\eqref{E:varphi psi almost id} implies that
\[
\psi_{\bar \rho}^{-1} = \Big( I_\infty + \gothD_{\bar \rho}^\infty \big(\widetilde \gothP_{\bar \rho}^\calB(\kappa) - \widetilde \gothP_{\bar \rho}^{\cl, \infty} \big) \gothC_{\bar \rho}^\infty\Big)^{-1} \circ \varphi_{\bar \rho} = (I_\infty+ \boldsymbol \epsilon) \circ \varphi_{\bar \rho}
\]
for the infinite matrix  $\boldsymbol \epsilon = \gothD_{\bar \rho}^\infty \big(\widetilde \gothP_{\bar \rho}^\calB(\kappa) - \widetilde \gothP_{\bar \rho}^{\cl, \infty} \big) \gothC_{\bar \rho}^\infty \in \Err_{\bar \rho}$.
\end{proof}

\begin{notation}
Fix $\bar \rho \in \scrB(U; \psi_m)$ a residual pseudo-representation.
Let $\HP_{\bar \rho, r}$ (resp. $\NP_{\bar \rho, r}$) denote the Hodge polygon (resp. Newton polygon) of the matrix $\gothU_p^{\cl, \bar \rho, r}$. Let $\HP_{\bar \rho, r}(i)$ (resp. $\NP_{\bar \rho, r}(i)$) denote the $y$-coordinate of the polygon when the $x$-coordinate is $i$.
Let $\alpha_0(\bar\rho)_r\leq \cdots \leq \alpha_{ d_{\bar \rho, r}-1}(\bar \rho)_r$ denote the slopes of $\HP_{\bar \rho, r}$ in non-decreasing order.
Let $\ord_{\bar \rho, r}$ denote the multiplicity of  the slope $0$ in $\NP_{\bar \rho, r}$.
Once again, we point out that the polygons $\HP_{\bar \rho, r}$ and $\NP_{\bar \rho, r}$ and hence the slopes $\alpha_0(\bar \rho)_r, \dots, \alpha_{d_{\bar \rho, r}-1}(\bar \rho)_r$ depend only on $r$ modulo $q$, as opposed to $r$.

Write the characteristic power series of $U_p$ on $ \widetilde \gothP_{\bar \rho}^\calB(\kappa) S^{D,\dagger}_{\calB}(U; \kappa)$ as
\[
\Char \big(U_p, \widetilde \gothP_{\bar \rho}^\calB(\kappa) S^{D,\dagger}_{\calB}(U; \kappa)\big) = 1 + c_{\bar \rho, 1}(w) X + c_{\bar \rho, 2}(w) X^2 + \cdots \in 1+ \calO\langle w \rangle \llbracket X \rrbracket.
\]
Its zero in $\calW(x\psi_m, p^{-1}) \times \GG_{m, \mathrm{rig}}$ is the spectral curve $\Spc_{\bar \rho}$ (over the weight disk $\calW(x\psi_m, p^{-1})$).
We have
\[
\Spc  \times_\calW \calW(x\psi_m; p^{-1}) = \bigcup_{\bar \rho \in \scrB(U; \psi_m)} \Spc_{\bar \rho}.
\]
\end{notation}

\begin{theorem}
\label{T:main theorem each residual}
Assume $m \geq 4$ as before.
 Theorem~\ref{T:sharp Hodge bound}, Corollary~\ref{C:precise NP computation}, and Theorem~\ref{T:improved main theorem} hold for each $\bar \rho \in \scrB(U;\psi_m)$, in the following sense.
\begin{enumerate}
\item
For any $w_0 \in \calW(x\psi_m, p^{-1})$, the Newton polygon of the power series $1+ c_{\bar \rho, 1}(w_0)X + \cdots $ lies above the  polygon starting at $(0,0)$  with slopes given by
\begin{equation}
\label{E:slopes of improved Hodge polygon bar rho}
\bigcup_{r=0}^\infty \big\{
\alpha_0(\bar \rho)_r+r,\alpha_1(\bar \rho)_r+r, \dots, \alpha_{\bar\rho, d_{\bar\rho,r}-1}(\bar \rho)_r+r\big\}.
\end{equation}

\item
For each $n \in \NN$, let $\lambda_{\bar \rho, n}$ denote the sum of $n$ smallest numbers in \eqref{E:slopes of improved Hodge polygon bar rho}.
Then
\[
c_{\bar \rho, n}(w) \in p^{\lambda_{\bar \rho, n}} \cdot \calO\langle w\rangle^\times, \quad \textrm{for all }n \textrm{ of the form }n=n_{\bar \rho, k}=\sum_{r=0}^k d_{\bar\rho, r}.
\]
In particular, for any $w_0 \in \calW(x\psi_m; p^{-1})$, the  Newton polygon of the power series $1+ c_{\bar \rho, 1}(w_0)X + \cdots $ passes through the point $(n, \lambda_{\bar \rho, n})$ for $n = n_{\bar \rho, k}$.

\item
Fix $r=0, \dots, q-1$.
Suppose that $(s_0, \NP_{\bar \rho, r}(s_0))$ is a vertex of the Newton polygon $\NP_{\bar \rho, r}$ and suppose that
\begin{equation}
\label{E:NP<HP+1 bar rho}
\NP_{\bar \rho, r}(s) < \HP_{\bar \rho, r}(s-1) + 1
\textrm{ for all } s = 1, \dots, s_0.
\end{equation}
Then for any $s=0, \dots, s_0$, any $n \in \ZZ_{\geq 0}$, and any $w_0 \in \calW(x\psi_m; p^{-1})$, the $(n_{\bar \rho, qn+r}+s)$th slope of the power series $1+ c_1(w_0)X + \cdots$ is the $s$th $U_p$-slope on $V(\bar \rho)_r$ plus $qn+r$.

\item

The spectral variety $\Spc_{\bar \rho}$ is a disjoint union of subvarieties
\[
X_{\bar \rho,0},\ X_{\bar \rho,(0,1]},\ X_{\bar \rho,(1,2]}, \ X_{\bar \rho,(2,3]},\ \dots
\]
such that each subvariety is finite and flat over $\calW(x\psi_m; p^{-1})$, and  for any closed point $x \in X_{\bar \rho,?}$, we have $v(a_p(x)) \in ?$.
Moreover, the degree of $X_{\bar \rho,(r,r+1]}$ over $\calW(x\psi_m; p^{-1})$ is exactly
\[
d_{\bar\rho, r+1} + \ord_{\bar \rho, r+1} - \ord_{\bar \rho, r}.
\]

\item
Keep the notation and hypothesis as in (3) and (4).
For all $n\in \ZZ_{\geq 0}$ and for a number  $\beta>0$ appearing in the first $s_0$ $U_p$-slopes on $V(\bar \rho)_r$, the closed points $x \in X_{ \bar \rho, (qn+r, qn+r+1]}$ for which $v(a_p(x)) = \beta + qn+r$ form a connected component of $X_{\bar \rho, (qn+r, qn+r+1]}$.  It is finite and flat over $\calW(x\psi_m; p^{-1})$ of degree equal to the multiplicity of $\beta$ in the set of $U_p$-slopes of $V(\bar \rho)_r$.
\end{enumerate}

\end{theorem}
\begin{proof}
By Proposition~\ref{P:identification of V(pi) and PS}, both $\varphi_{\bar \rho}$ and $\psi_{\bar \rho}$ are isomorphisms of Banach spaces.
So we have
\[
\Char\big(U_p;  \widetilde \gothP_{\bar \rho}^\calB(\kappa) S^{D,\dagger}_{\calB}(U; \kappa) \big) = \Char\big( (\psi_{\bar \rho})^{-1} \circ \gothU_p^\calB\circ \psi_{\bar \rho}; V({\bar \rho})_A^\infty\big).
\]

Recall from Proposition~\ref{P:Tl Up overconvergent equiv classical}(2) that the infinite block diagonal matrix $\gothU_p^{\cl, \infty} = \Diag \{\gothU_p^\cl(\psi_m),\, p\cdot  \gothU_p^\cl(\psi_m \omega^{-2}),\, p^2\cdot \gothU_p^\cl(\psi_m \omega^{-4}),\, \dots\}$ satisfies
\[
\gothU_p^\calB(\kappa) - \gothU_p^{\cl, \infty} \in \textrm{ the error space }\Err_p \textrm{ in \eqref{E:Errp}}.
\]

We introduce the following error space
\[
\Err_{\bar \rho, p}: =
\begin{pmatrix}
p\rmM_{d_{\bar \rho, 0}}(A^\circ) & p\rmM_{d_{\bar \rho, 0}\times d_{\bar \rho, 1}}(A^\circ) & p^2 \rmM_{d_{\bar \rho, 0}\times d_{\bar \rho, 2}}(A^\circ) & p^3\rmM_{d_{\bar \rho, 0}\times d_{\bar \rho, 3}}(A^\circ) &\cdots
\\
p^3\rmM_{d_{\bar \rho, 1}\times d_{\bar \rho, 0}}(A^\circ) &  p^2 \rmM_{d_{\bar \rho, 1}}(A^\circ) & p^2\rmM_{d_{\bar \rho, 1}\times d_{\bar \rho, 2}}(A^\circ) &p^3 \rmM_{d_{\bar \rho, 1}\times d_{\bar \rho, 3}}(A^\circ) &  \cdots
\\
p^4  \rmM_{d_{\bar \rho, 2}\times d_{\bar \rho, 0}}(A^\circ) &p^4 \rmM_{d_{\bar \rho, 2}\times d_{\bar \rho, 1}}(A^\circ) & p^3 \rmM_{d_{\bar \rho, 2}\times d_{\bar \rho, 2}}(A^\circ)& p^3\rmM_{d_{\bar \rho, 2}\times d_{\bar \rho, 3}}(A^\circ) & \cdots\\
p^5\rmM_{d_{\bar \rho, 3}\times  d_{\bar \rho, 0}}(A^\circ) &p^5\rmM_{d_{\bar \rho, 3}\times  d_{\bar \rho, 1}}(A^\circ) &p^5\rmM_{d_{\bar \rho, 3}\times d_{\bar \rho, 2}}(A^\circ) &  p^4\rmM_{d_{\bar \rho, 3}}(A^\circ)&  \cdots
\\
\vdots & \vdots & \vdots & \vdots & \ddots
\end{pmatrix},
\]
where the $(i,j)$-block entry is
\begin{itemize}
\item
$ p^{i+1}\rmM_{d_{\bar \rho,i}}(A^\circ)$ if $i=j$,
\item
$p^{i+2}\rmM_{d_{\bar \rho, i}\times d_{\bar \rho, j}}(A^\circ)$ if $i > j$, and
\item
$p^{j}\rmM_{d_{\bar \rho, i}\times d_{\bar \rho, j}}(A^\circ)$ if $i < j$.
\end{itemize}
Rewrite the composite $(\psi_{\bar \rho})^{-1} \circ \gothU_p^\calB \circ  \psi_{\bar{\rho}}$ as
\begin{align}
\nonumber
(\psi_{\bar \rho})^{-1}& \circ \gothU_p^\calB \circ  \psi_{\bar \rho}
=
(\mathbf{id} + \boldsymbol \epsilon)
\gothD_{\bar \rho}^\infty  \widetilde \gothP_{\bar \rho}^\calB(\kappa) \gothU_p^\calB \gothC_{\bar \rho}^\infty
\\
\nonumber
&= \gothD_{\bar \rho}^\infty    \widetilde \gothP_{\bar \rho}^\calB(\kappa)
\gothU_p^\calB\gothC_{\bar \rho}^\infty
+
\boldsymbol \epsilon
\gothD_{\bar \rho}^\infty  \widetilde \gothP_{\bar \rho}^\calB(\kappa) \gothU_p^\calB \gothC_{\bar \rho}^\infty \\
\label{E:last estimate}
&=\gothD_{\bar \rho}^\infty  \widetilde \gothP_{\bar \rho}^{\cl, \infty} \gothU_p^{\cl, \infty}  \gothC_{\bar \rho}^\infty
+
 \gothD_{\bar \rho}^\infty ( \widetilde \gothP_{\bar \rho}^\calB(\kappa) \gothU_p^\calB - \widetilde \gothP_{\bar \rho}^{\cl, \infty} \gothU_p^{\cl, \infty})  \gothC_{\bar \rho}^\infty
+\boldsymbol \epsilon \gothD_{\bar \rho}^\infty\widetilde \gothP_{\bar \rho}^\calB (\kappa) \gothU_p^\calB  \gothC_{\bar \rho}^\infty .
\end{align}
Here the second equality in the first line follows from the commutativity of $\widetilde \gothP_{\bar \rho}^\calB(\kappa)$ and $ \gothU_p^\calB$ as they are (limits of) Hecke operators.
It suffices to understand each of the terms.
\begin{enumerate}
\item[(i)]
The first term $\gothD_{\bar \rho}^\infty\widetilde \gothP_{\bar \rho}^{\cl, \infty} \gothU_p^{\cl, \infty}    \gothC_{\bar \rho}^\infty$ of \eqref{E:last estimate} exactly gives the action of $U_p$ on the space of classical automorphic forms.
\item[(ii)]
By Proposition~\ref{P:Tl Up overconvergent equiv classical}, we easily deduce that
\begin{align*}
 \widetilde \gothP_{\bar \rho}^\calB (\kappa)
\gothU_p^\calB& - \widetilde \gothP_{\bar \rho}^{\cl, \infty} \gothU_p^{\cl, \infty}  = \widetilde \gothP_{\bar \rho}^\calB(\kappa) (\gothU_p^\calB - \gothU_p^{\cl, \infty})
+  ( \widetilde \gothP_{\bar \rho}^\calB(\kappa) - \widetilde \gothP_{\bar \rho}^{\cl, \infty})\gothU_p^{\cl, \infty}
\\
& \in  (\widetilde \gothP_{\bar \rho}^{\cl, \infty}  + \Err)\cdot \Err_p + \Err \cdot \gothU_p^{\cl, \infty}  \ \subseteq  \Err_p;
\end{align*}
so the middle term of \eqref{E:last estimate}
\[
\gothD_{\bar \rho}^\infty ( \widetilde \gothP_{\bar \rho}^\calB(\kappa) \gothU_p^\calB  - \widetilde \gothP_{\bar \rho}^{\cl, \infty} \gothU_p^{\cl, \infty})  \gothC_{\bar \rho}^\infty \in \gothD_{\bar \rho}^\infty\cdot \Err_p \cdot \gothC_{\bar \rho}^\infty \subseteq \Err_{{\bar \rho}, p}.
\]
\item[(iii)]
We write
\[\boldsymbol \epsilon
\gothD_{\bar \rho}^\infty \widetilde \gothP_{\bar \rho}^\calB(\kappa) \gothU_p^\calB  \gothC_{\bar \rho}^\infty =\boldsymbol \epsilon \gothD_{\bar \rho}^\infty (\widetilde \gothP_{\bar \rho}^\calB(\kappa) \gothU_p^\calB  -\gothP_{\bar \rho}^{\cl,\infty} \gothU_p^{\cl, \infty}) \gothC_{\bar \rho}^\infty  +\boldsymbol \epsilon \gothD_{\bar \rho}^\infty\gothP_{\bar \rho}^{\cl,\infty}  \gothU_p^{\cl, \infty} \gothC_{\bar \rho}^\infty.
\]
The second term belongs to $\Err_{{\bar \rho}, p}$ because
 $\boldsymbol \epsilon \in \Err_{\bar \rho}$ by Proposition~\ref{P:identification of V(pi) and PS}.
For the first term, we use the argument in (ii) to see that it belongs to
\[ \boldsymbol \epsilon \cdot
\gothD_{\bar \rho}^\infty\cdot \Err_p \cdot\gothC_{\bar \rho}^\infty\ \subseteq \Err_{{\bar \rho}, p}.
\]

\end{enumerate}
Combining the computation above, we see that $(\psi_{\bar \rho})^{-1} \circ \gothU_p^\calB\circ \psi_{\bar \rho}$ belongs to
\begin{equation}
\label{E:Up on bar rho}
\begin{pmatrix}
 \gothU_p^{\cl, \bar\rho, 0} &&&\\
& p\cdot \gothU_p^{\cl, \bar\rho, 1} && \\
&&p^2\cdot \gothU_p^{\cl, \bar\rho, 2} & \\
 &  & &\ddots
\end{pmatrix} + \Err_{\bar \rho, p}.
\end{equation}

At this point, (1)--(4) of the Theorem can be proved in the same way as they were proved in Theorem~\ref{T:sharp Hodge bound}, Corollary~\ref{C:precise NP computation}, and Theorem~\ref{T:improved main theorem}, with the modifications indicated below.

(1) already follows from the estimate \eqref{E:Up on bar rho} because each $\gothU_p^{\cl, \bar \rho, r}$ is already written in the form adapted to its Hodge polygon.

For (2), we need to consider the action of $P_{\bar \rho}$ on the space $S_2^D(U; \psi_{m,w}\omega^{-2r})$ (see \eqref{E:S2 deformation} for the definition).
Let $\widetilde P_{\bar \rho}$ denote the limit $\lim_{n \to \infty} (P_{\bar \rho})^{p^n}$.
By the same argument as in Proposition~\ref{P:decomposition}, we have $\widetilde P_{\bar\rho}^2 = \widetilde P_{\bar \rho}$, $ \widetilde P_{\bar \rho}  \widetilde P_{\bar \rho'} = 0$ for $\bar \rho \neq \bar \rho'$, and $\sum_{ \bar\rho \in \scrB(U; \psi_m)}  \widetilde P_{\bar \rho} = \mathrm{id}$.
We use $V(\bar \rho, w)_r$ to denote the image $ \widetilde P_{\bar \rho} S_2^D(U; \psi_{m,w}\omega^{-2r})$, which is isomorphic to $V(\bar\rho)_r \otimes_\calO \calO/p^2\calO[w]$, as an $\calO$-module.
Let $\gothU_p^{\cl, \bar \rho, w, r}$ denote the matrix for the $U_p$-action on $V(\bar \rho, w)_r$ with respect to the basis $e_0(\bar\rho)_r, \dots, e_{d_{\bar\rho, r}-1}(\bar \rho)_r$; its $i$th row is divisible by $p^{\alpha_i(\bar \rho)_r}$, and all coefficients on $w$ belongs to $p\calO/ p^2\calO$.
We use
$\overline \gothU_p^{\cl, \bar\rho,w, r}$ to denote the matrix given by dividing the $i$th row of $\gothU_p^{\cl, \bar \rho,w, r}$ by $p^{\alpha_i(\bar \rho)_r}$.
As argued in the proof of Theorem~\ref{T:sharp Hodge bound}(2), it suffices to prove that $\det \overline \gothU_p^{\cl, \bar\rho,w, r}$ belongs to $\FF^\times \subseteq \FF[w]$ for each $r$.
However, this follows from the fact that the product
\[
\prod_{\bar\rho \in \scrB(U; \psi_m)}
\det \overline \gothU_p^{\cl, \bar\rho,w, r} = \det \overline \gothU_p^{\cl ,\bfe}(\psi_{m,w}\omega^{-2r}) \in \FF^\times.
\]

(3) and (4) follow from the arguments in Corollary~\ref{C:precise NP computation} and Theorem~\ref{T:improved main theorem} with no essential changes.
(5) follows from (3) immediately.
\end{proof}

\end{document}